\documentclass{amsart}
\usepackage[T1]{fontenc}
\usepackage[english,french]{babel}

\usepackage{amsmath}
\usepackage{amssymb}
\usepackage{rotating}
\usepackage{graphicx}
\usepackage{url}
\usepackage{array}
\usepackage{tikz-cd}

\newtheorem{theo}{Theorem}[section]
\newtheorem{lemm}[theo]{Lemme}
\newtheorem{prop}[theo]{Proposition}
\newtheorem{coro}[theo]{Corollaire}

\theoremstyle{remark}
\newtheorem{rema}[theo]{Remarque}
\newtheorem{enonce}[theo]{Remarques}
\newtheorem{defi}[theo]{D\'efinition}
\newtheorem*{remerc}{Remerciements}

\numberwithin{equation}{section}

\newcommand\case[1]{{\ensuremath{{\mathrm(}{\mathit{#1\/}}{\mathrm)}}}}
\newcommand\subsecno{{$\mathrm{n}^{\mathrm{o}}$}}
\newcommand\ve{\ensuremath{\varepsilon}}
\newcommand\sF{\ensuremath{\mathsf{F}}}
\newcommand\se{\ensuremath{\mathsf{e}}}
\newcommand\cC{\ensuremath{\mathcal{C}}}
\newcommand\cD{\ensuremath{\mathcal{D}}}
\newcommand\cE{\ensuremath{\mathcal{E}}}
\newcommand\cO{\ensuremath{\mathcal{O}}}
\newcommand\cP{\ensuremath{\mathcal{P}}}
\newcommand\cU{\ensuremath{\mathcal{U}}}
\newcommand\cW{\ensuremath{\mathcal{W}}}
\newcommand\cR{\ensuremath{\mathcal{R}}}
\newcommand\cN{\ensuremath{\mathcal{N}}}
\newcommand\cS{\ensuremath{\mathcal{S}}}
\newcommand\cK{\ensuremath{\mathcal{K}}}
\newcommand\cM{\ensuremath{\mathcal{M}}}
\newcommand\cI{\ensuremath{\mathcal{I}}}
\newcommand\cJ{\ensuremath{\mathcal{J}}}
\newcommand\bN{\ensuremath{\mathbb{N}}}
\newcommand\bZ{\ensuremath{\mathbb{Z}}}
\newcommand\bQ{\ensuremath{\mathbb{Q}}}
\newcommand\bR{\ensuremath{\mathbb{R}}}
\newcommand\bC{\ensuremath{\mathbb{C}}}
\newcommand\bA{\ensuremath{\mathbb{A}}}
\newcommand\bP{\ensuremath{\mathbb{P}}}

\newcommand\bG{\ensuremath{\mathbb{G}}}
\newcommand\ga{\ensuremath{\mathfrak{a}}}
\newcommand\gA{\ensuremath{\mathfrak{A}}}
\newcommand\gb{\ensuremath{\mathfrak{b}}}
\newcommand\gc{\ensuremath{\mathfrak{c}}}

\newcommand\gp{\ensuremath{\mathfrak{p}}}
\newcommand\gP{\ensuremath{\mathfrak{P}}}
\newcommand\gQ{\ensuremath{\mathfrak{Q}}}

\DeclareMathOperator\id{id}
\DeclareMathOperator\ord{ord}
\DeclareMathOperator\Spec{Spec}
\DeclareMathOperator\PGL{PGL}
\DeclareMathOperator\GL{GL}
\DeclareMathOperator\Hom{Hom}
\DeclareMathOperator\Div{Div}
\DeclareMathOperator\Pic{Pic}
\DeclareMathOperator\GSp{GSp}
\DeclareMathOperator\N{N}
\DeclareMathOperator\tr{tr}
\DeclareMathOperator\cm{cm}
\DeclareMathOperator\nr{nr}
\DeclareMathOperator\ab{ab}
\DeclareMathOperator\myll{ll}
\DeclareMathOperator\Gal{Gal}
\DeclareMathOperator\Orb{Orb}
\DeclareMathOperator\Sp{Sp}

\title[Croissance asymptotique de nombres de Weil]{Croissance asymptotique de nombres de Weil appartenant \`a un corps de nombres fix\'e}
\date{le \today}
\author{John Boxall}

\address{Laboratoire de Math\'ematiques Nicolas Oresme, UMR CNRS 6139, Campus 2, Universit\'e de Caen-Normandie, 14032 Caen cedex, France}
\email{john.boxall@unicaen.fr}


\begin{document}

\begin{abstract}
Nous \'etablissons une \'equivalence asymptotique lorsque $x\to +\infty$ pour le nombre d'entiers alg\'ebriques $\alpha$ appartenant \`a un corps de nombres CM $K$ donn\'e tels que $\alpha\overline{\alpha}\in \bN$ et $\alpha\overline{\alpha}\leq x$. Ce prob\`eme est li\'e \`a la fonction z\^eta des hauteurs $Z_h(X^K,s)$ associ\'ee \`a la classe anticanonique d'une certaine vari\'et\'e torique $X^K$ sur $\bQ$ et nous montrons que $Z_h(X^K,s)$ poss\`ede un prolongement m\'eromorphe dans le demi-plan $\{\Re(s)>\frac{1}{2}\}$, holomorphe sauf en $s=1$. Nous obtenons au passage une nouvelle d\'emonstration de la conjecture de Manin concernant la croissance asymptotique des points de hauteur born\'ee de $X^K(\bQ)$.   
\medskip

We prove an asymptotic formula as $x\to +\infty$ for the number of algebraic intagers $\alpha$ belonging to a given CM number field $K$ with $\alpha\overline{\alpha}\in \bN$ and $\alpha\overline{\alpha}\leq x$ This problem is related to the height zeta function $Z_h(X^K,s)$ associated to the anticanonical class of a certain toric variety $X^K$ over $\bQ$  and we show that $Z_h(X^K,s)$ has a meromorphic continuation to the half-plane $\{\Re(s)>\frac{1}{2}\}$ where it is holomorphic except at $s=1$. Along the way we obtain a new proof of Manin's conjecture on the asymptotic growth of points on $X^K(\bQ)$ of bounded height.    
\end{abstract}

\maketitle

\section{Introduction} 

\subsection{\'Enonc\'e des r\'esultats principaux} Soit $F$ une extension alg\'ebrique de $\bQ$ et soit $n\in \bN^*$. Un \'el\'ement $\alpha\in F$ est dit un \emph{$n$-nombre de Weil} s'il est entier sur $\bZ$ et si pour tout plongement $\phi:F\to \bC$, on a $\phi(\alpha)\overline{\phi(\alpha)}=n$. Notons $\cW_F(n)$ l'ensemble des $n$-nombres de Weil appartenant \`a $F$ et posons $\cW_F=\bigcup_{n\geq 1}\cW_F(n)$. Si $\alpha\in \cW_F$, on note $\alpha\overline{\alpha}$ la valeur commune des $\phi(\alpha)\overline{\phi(\alpha)}$, o\`u $\phi\in \Hom(F,\bC)$.

Lorsque $F$ est un corps de nombres (c'est-\`a-dire une extension finie de $\bQ$), les ensembles $\cW_F(n)$ sont finis. Notons alors $a_F(n)$ le cardinal de $\cW_F(n)$. Le but de cet article est d'\'etudier le comportement asymptotique lorsque $x\to +\infty$ de la somme
\begin{equation*}
\cN(F,x)=\sum_{1\leq n\leq x}a_F(n).
\end{equation*}
Voici une premi\`ere version de notre r\'esultat principal. Rappelons qu'un \emph{corps CM} est un corps de nombres extension quadratique totalement imaginaire d'un corps totalement r\'eel (voir la \S~\ref{sec:CMrappels}). 

\begin{theo}\label{theo:premvers}
Soit $F$ un corps de nombres contenant un corps CM. Alors il existe une constante $c_F>0$ et un entier $\rho_F\geq 1$ ne d\'ependant que de $F$ tels que 
\begin{equation*}
\cN(F,x)\sim  c_Fx(\log{x})^{\rho_F-1},   \qquad  x\to +\infty.
\end{equation*} 
\end{theo}

Donnons quelques indications concernant la strat\'egie de la d\'emonstration de ce r\'esultat. Soit $K$ le sous-corps CM maximal de $F$ (voir le corollaire~\ref{coro:maxtrcm}). Alors $\cN(F,x)-\cN(K,x)=O(\sqrt{x})$ (voir le th\'eor\`eme~\ref{theo:wFnfini}) et, par cons\'equent, une fois le th\'eor\`eme \'etabli pour $K$, on le d\'eduit pour $F$ et en m\^eme temps que $c_F=c_K$ et $\rho_F=\rho_K$. 

Notons donc d\'esormais $K$ un corps CM et posons
\begin{equation*}
Z_0(K,s)=\sum_{n\geq 1}a_K(n)n^{-s}, \qquad \Re(s)\gg 0.
\end{equation*} 
En fait nous verrons que la s\'erie converge pour tout $s\in \bC$ avec $\Re(s)>1$. Notons $\zeta(s)$ la fonction z\^eta de Riemann.

Posons $g=\frac{1}{2}[K:\bQ]$ et notons $K^{w,\times}$ le groupe $\{\alpha\in K^\times\mid \alpha\overline{\alpha}\in \bQ^\times\}$. On associe alors \`a $K$ un tore $W^K$ sur $\bQ$ de dimension $g+1$ ayant la propri\'et\'e que $W^K(\bQ)$ s'identifie avec $K^{w,\times}$. Ce tore se plonge dans l'espace affine $\bA^{2g}_\bQ$ d\'epourvu de l'origine et on note $V^K$ son image dans $\bP^{2g-1}_{\bQ}$. Autrement dit, $V^K$ est d\'efini par une suite exacte de $\bQ$-tores $1\to \bG_{m,\bQ}\to W^K\to V^K\to 1$. D'apr\`es le th\'eor\`eme $90$ de Hilbert, $V^K(\bQ)$ est isomorphe au groupe quotient $K^{w,\times}/\bQ^\times$. Notons $\cO_K$ l'anneau des entiers de $K$. Un \'el\'ement non-nul de $K$ est dit \emph{r\'eduit} s'il appartient \`a $\cO_K$ mais n'est divisible par aucun entier naturel $>1$. On voit ais\'ement que tout point de $V^K(\bQ)$ poss\`ede un repr\'esentant r\'eduit appartenant \`a $K^{w,\times}$ qui est unique au signe pr\`es.
  
\begin{prop}\label{prop:compaqVK} Le tore $V^K$ poss\`ede une compactification \'equivariante lisse $X^K$ ayant la propri\'et\'e suivante. 
Soit $P\in V^K(\bQ)$ et soit $\alpha\in \cO_K$ un repr\'esentant r\'eduit de $P$. Si en plus $\alpha$ n'est divisible par aucune place ramifi\'ee de $K/\bQ$, alors la hauteur de $P$ par rapport au diviseur anticanonique de $X^K$ (plong\'ee convenablement dans un espace projectif) est \'egale \`a $\alpha\overline{\alpha}$.   
\end{prop} 

Or, il s'av\`ere que fonction $\zeta(2s)^{-1}Z_0(K,s)$ est li\'ee \`a la \emph{fonction z\^eta des hauteurs} (au sens de \cite{FrMaTsch89}) de $X^K$ associ\'ee \`a son diviseur anticanonique. Notons cette fonction $Z_h(X^K,s)$. Dans certains cas on peut ainsi ramener le th\'eor\`eme~\ref{theo:premvers} \`a un cas particulier de la conjecture de Manin qui pr\'edit l'\'equivalence asymptotique \'enonc\'ee dans le th\'eor\`eme~\ref{theo:premvers} avec l'entier $\rho_K$ \'egal au rang du sous-groupe $\Pic(X^K)_{\bQ}$ du groupe de Picard $\Pic(X^K)$ de $X^K$ form\'e des \'el\'ements fix\'es par le groupe de Galois absolu $\Gal(\overline{\bQ}/\bQ)$ de $\bQ$. Tel est le cas lorsque le groupe $\cC_K^w$ qui sera d\'efinie au d\'ebut de la sous-section~\ref{subsect:stratdemo} est trivial, auquel cas nous verrons que les s\'eries de Dirichlet $\frac{1}{2}\zeta(2s)^{-1}Z_0(K,s)$ et $Z_h(X^K,s)$ poss\`edent toutes les deux des produits eul\'eriens dont les facteurs sont les m\^emes en dehors des nombres premiers ramifi\'es dans $K$. Dans son travail Peyre~\cite{Peyre95} a propos\'e, sous certaines conditions qui sont satisfaites notamment par les vari\'et\'es toriques projectives lisses (voir \cite{Peyre01}, \S~2.1), une formule explicite pour la constante $c_K$.

La conjecture de Manin (ainsi que son raffinement par Peyre) ayant \'et\'e d\'emontr\'e dans le cas des vari\'et\'es toriques par Batyrev et Tschinkel \cite{BatTsch95}, \cite{BatTsch98} (voir aussi le travail de Bourqui \cite{Bourqui11}), le th\'eor\`eme~\ref{theo:premvers} devient dans le cas o\`u $\cC_K^{w}$ est trivial un simple corollaire de leurs r\'esultats. En fait, il r\'esulte de leur travail qu'il existe $\ve>0$ tel que $Z_h(K,s)$ se prolonge en une fonction m\'eromorphe sur le demi-plan $\{\Re(s)>1-\ve\}$, holomorphe \`a l'exception d'un unique p\^ole d'ordre $\rho_K$ en $s=1$.  Le th\'eor\`eme~\ref{theo:premvers} en d\'ecoule par application d'un th\'eor\`eme taub\'erien;  en fait, nous sommes en mesure de d\'emontrer un r\'esultat plus pr\'ecis (voir le th\'eor\`eme~\ref{theo:asymptpol}).

Toutefois, pour d\'emontrer le th\'eor\`eme en toute g\'en\'eralit\'e, nous sommes oblig\'es d'adopter une autre strat\'egie, qui permet en m\^eme temps d'\'etendre le domaine de m\'eromorphie de $Z_0(K,s)$ et de $Z_h(K,s)$ \`a un demi-plan plus grand.  Le th\'eor\`eme suivant r\'esume une partie des conclusions des Propositions~\ref{prop:Z0prlong}, \ref{prop:ZhManin} et \ref{prop:g2prolong}:---  

\begin{theo}\label{theo:prolmer}
Soit $K$ un corps CM. On pose $g=\frac{1}{2}[K:\bQ]$. Les fonctions $Z_0(K,s)$ et $Z_h(K,s)$ poss\`edent des prolongements m\'eromorphes sur le demi-plan $\{\Re(s)>\frac{1}{2}\}$. Elles sont holomorphes en dehors de $s=1$, o\`u elles ont un p\^ole d'ordre \'egal au rang de $\Pic(X^K)_{\bQ}$. Si $g=1$, elles se prolongent en des fonctions holomorphes sur $\bC- \{1\}$. Si $g=2$, elles se prolongent en des fonctions m\'eromorphes sur $\bC$.  
\end{theo}

Lorsque $g=1$, $K$ est un corps quadratique imaginaire et, si $w_K$ d\'esigne l'ordre du groupe des racines de l'unit\'e dans $K$, alors
\begin{equation*}
Z_0(K,s)=w_K\zeta_{K,0}(s)=w_K\sum_{(\alpha)}\N_{K/\bQ}(\alpha)^{-s},
\end{equation*} 
o\`u $\zeta_{K,0}(s)$ est la fonction z\^eta partielle de $K$ associ\'ee \`a la classe des id\'eaux principaux de $K$, la somme parcourant l'ensemble $(\alpha)$ des id\'eaux principaux de l'anneau des entiers $\cO_K$ de $K$ et $\N_{K/\bQ}(\alpha)$ d\'esignant l'indice du groupe additif $\alpha\cO_K$ dans $\cO_K$. L'holomorphie de $Z_0(K,s)$ dans le cas $g=1$ est donc un r\'esultat classique, qui peut \^etre d\'eduite de celles de certaines fonctions $L$ en \'ecrivant $\zeta_{K,0}(s)$ sous la forme $\frac{1}{h_K}\sum_{\chi}{L(s,\chi)}$, o\`u $h_K$ est l'ordre du groupe de Picard de $\cO_K$ et la somme parcourt les fonctions $L$ attach\'ees aux caract\`eres $\chi$ de $\Pic(\cO_K)$ identifi\'e par le biais de la loi de r\'eciprocit\'e d'Artin avec le groupe de Galois relatif \`a $K$ du corps de classes de Hilbert de $K$. 

\subsection{Strat\'egie de la d\'emonstration du th\'eor\`eme~\ref{theo:prolmer}}\label{subsect:stratdemo} Notre strat\'egie pour d\'emontrer le th\'eor\`eme~\ref{theo:prolmer} lorsque $g\geq 2$ s'inspire de celle que nous venons d'esquisser dans le cas $g=1$. Si $F$ est un corps de nombres, on note $\cO_F$ son anneau des entiers, $I_F$ le groupe des id\'eaux inversibles (c'est-\`a-dire $\cO_F$-sous-modules de $F$ projectifs de rang un) et $P_F$ le sous-groupe de $I_F$ form\'e des id\'eaux principaux. Le groupe de Picard $\Pic(\cO_F)$
s'identifie alors avec le groupe quotient $I_F/P_F$. Il est bien connu que $\Pic(\cO_F)$ est un groupe fini. On note $w_F$ l'ordre du sous-groupe de torsion du groupe multiplicatif $F^\times$. 

Rappelons que $K$ d\'esigne un corps CM. On note 
\begin{itemize}
\item $K^{\times, w}$ le groupe $\{\alpha\in K^\times \mid \alpha\overline{\alpha}\in \bQ^\times\}$, 

\item $I_K^w$ le sous-groupe de $I_K$ form\'es des id\'eaux fractionnaires $\ga$ pour lesquels il existe $q\in \bQ^\times$ tel que $\ga\overline{\ga}=q\cO_K$,

\item $P_K^w$ le sous-groupe de $I_K^w$ form\'e des id\'eaux engendr\'es par un \'el\'ement de $K^{\times, w}$,

\item $\cC_K^w$ le groupe quotient $I_K^w/P_K^w$,
\end{itemize}  

Nous verrons que $\cC_K^w$ est un groupe fini (voir la proposition~\ref{prop:CwKfini}). Notons $h^w_K$ son ordre. 

Notons $I^{w,+}_K$ sous-mono\"{\i}de multiplicatif de $I^{w}_K$ form\'e des id\'eaux entiers. Pour tout caract\`ere $\chi$ de $\cC^{w}_K$, on pose 
\begin{equation*}
 Z(K,\chi,s)=\sum_{\ga\in I^{w,+}_K}\chi(\ga)\N_{K/\bQ}(\ga)^{-s/g}, \qquad Z(K,s)=Z(K,1,s),
\end{equation*}
o\`u $\N_{K/\bQ}(\ga)$ d\'esigne l'indice du groupe additif $\ga$ dans $\cO_K$ et $\chi(\ga)$ est la valeur de $\chi$ en l'image de $\ga$ dans $\cC^{w}_K$. (Remarquons que si $\ga\overline{\ga}=n\cO_K$ avec $n\in \bN^*$, alors $\N_{K/\bQ}(\ga)=n^g$.) On a alors
\begin{equation*}
Z_0(K,s)=\frac{w_K}{h^w_K}\sum_{\chi}Z(K,\chi,s).   \tag*{$(*)$}
\end{equation*}

Or, l'avantage des fonctions $Z(K,\chi,s)$ est qu'elles poss\`edent des produits eul\'eriens
\begin{equation*}
Z(K,\chi,s)=\prod_{p}Z_p(K,\chi,s), \qquad  Z(K,s)=\prod_{p}Z_p(K,s)
\end{equation*}
--- o\`u $p$ parcourt l'ensemble des nombres premiers --- et que les coefficients des facteurs $Z_p(K,\chi,s)$ (et en particulier $Z_p(K,s)$) se calculent en fonction de la forme de la d\'ecomposition de $p\cO_K$ comme produit d'id\'eaux premiers. Cela permet de montrer que $Z(K,\chi, s)$ se prolonge m\'eromorphiquement sur $\{\Re(s)>\frac{1}{2}\}$ avec un unique p\^ole d'ordre au plus $\rho_K$ en $s=1$, l'ordre pr\'ecis du p\^ole \'etant determin\'e en fonction de $\chi$. En vertu de $(*)$, on d\'eduit le prolongement m\'eromorphe de $Z_0(K,s)$ sur $\{\Re(s)>\frac{1}{2}\}$ avec unique p\^ole d'ordre au plus $\rho_K$ en $s=1$, et on montre que l'ordre du p\^ole est, en fait, \'egal \`a $\rho_K$ (voir le lemme~\ref{lemm:Z0KminZK}). 

En appliquant une l\'eg\`ere g\'en\'eralisation de la proposition~\ref{prop:compaqVK}, nous verrons que la fonction z\^eta des hauteurs $Z^h_0(X^K,s)$ de $X^K$ peut s'\'ecrire sous une forme analogue \`a $(*)$:
\begin{equation*}
Z^h_0(X^K,s)=\frac{w_K}{2h^w_K\zeta(2s)}\sum_{\chi}\tilde{Z}(K,\chi,s),
\end{equation*} 
o\`u, pour chacun des caract\`eres $\chi$, $\tilde{Z}(K,\chi,s)$ poss\`edent un produit eul\'erien qui diff\`ere de celui de $Z(K,\chi,s)$ seulement en les facteurs aux nombres premiers ramifi\'es dans $K$. On obtient ainsi une nouvelle d\'emonstration de la conjecture de Manin pour la vari\'et\'e torique $X^K$. 

Notons $L$ la cl\^oture galoisienne de $K$ dans $\bC$ et $\Gamma$ le groupe de Galois de $L$ sur $\bQ$. Notons $\cC\cM_K$ l'ensemble des types CM de $K$ (voir le \subsecno~\ref{subsection:typesCM}); il suffit de rappeler ici qu'un type CM sur $K$ est un sous-ensemble $\Phi$ de $\Hom(K,\bC)$ tel que $\Phi\cup \overline{\Phi}=\Hom(K,\bC)$ et $\Phi\cap \overline{\Phi}=\emptyset$. Il est clair que $\Hom(K,\bC)$ s'identifie avec $\Hom(K,L)$. En particulier, $\Gamma$ op\`ere sur $\cC\cM_K$: si $\gamma\in \Gamma$ et si $\Phi\in \cC\cM_K$, alors $\gamma\Phi:=\{\gamma\circ\phi\mid \phi\in \Phi\}\in \cC\cM_K$. La proposition suivante r\'esume le lien entre les propri\'et\'es arithm\'etiques de $X^K$ et celles du corps de nombres $K$ (voir le th\'eor\`eme~\ref{theo:modelesDVX} et la proposition~\ref{prop:Frob}). 

\begin{prop} \label{prop:XKVKcle} La vari\'et\'e torique $X^K$ se plonge comme intersection compl\`ete lisse de $g-1$ quadriques dans l'espace projectif $\bP_\bQ^{2g-1}$ de telle mani\`ere que $X^K-V^K$ se d\'ecompose, apr\`es extension de scalaires \`a $L$, comme r\'eunion d'un ensemble $\cP$ de $2^g$ espaces projectifs de dimension $g-1$. En outre, les $\Gamma$-ensembles $\cP$ et $\cC\cM_K$ sont isomorphes.   
\end{prop}

Cette proposition permet notamment de calculer la fonction $L$ d'Artin associ\'ee \`a $\Pic(V^K)_\bQ$. Plus g\'en\'eralement on construit, pour tout caract\`ere $\chi$ de $\cC^w_K$, un produit fini $\Pi(K,\chi,s):=\prod_{\Phi\in \Orb^*(\cC\cM_K)}{L(\chi_{\hat{\Phi}},s)}$ de fonctions $L$ de repr\'esentations ab\'eliennes partout non-ramifi\'ees des groupes de Galois de certains sous-corps $\hat{K}_\Phi$ de $L$  (on renvoie \`a la discussion pr\'ec\'edant le corollaire~\ref{coro:quothol} pour les notations). Ce produit a la propri\'et\'e que $Z(K,\chi,s)/\Pi(K,\chi,s)$ s'\'ecrit (apr\`es suppression des facteurs eul\'eriens aux nombres premiers ramifi\'es dans $K$) comme un produit eul\'erien convergent sur $\{\Re(s)>\frac{1}{2}\}$. Cela permet d'identifier l'ordre du p\^ole de $Z(K,\chi,s)$ en $1$ et d'en d\'eduire le th\'eor\`eme~\ref{theo:prolmer}. 

Lorsque $g=2$, on montre que les quotients $Z(K,\chi,s)/\Pi(K,\chi,s)$ s'ex\-priment eux-m\^emes (aux facteurs eul\'eriens aux nombres premiers ramifi\'es pr\`es) comme s\'eries $L$ (en la variable $2s$), et le prolongement m\'eromorphe sur $\bC$ de $Z(K,\chi, s)$ en d\'ecoule imm\'ediatement (voir la \S~\ref{sec:gendeux}). Par contre, lorsque $g\geq 3$, vus les r\'esultats de Kurokawa  \cite{Kurok86, Kurok87}, il para\^{\i}t peu probable que ces fonctions puissent \^etre prolong\'ees sur $\bC$ tout entier. 

Les fonctions $\Pi(K,\chi,s)$ \'etant des produits de fonctions $L$ de Hecke, des r\'esultats sur la croissance dans des bandes verticales de ce type de fonction (voir par exemple \cite{IwKow04}, Chapter V),  l'application d'un th\'eor\`eme taub\'erien --- par exemple celui d\'emontr\'e dans l'appendice~A de \cite{ChLTsch01} --- permet de d\'emontrer le r\'esultat suivant plus fort que le th\'eor\`eme~\ref{theo:premvers}. 

\begin{theo} \label{theo:asymptpol} Il existe $P(X)\in \bR[X]$ de degr\'e $\rho_F-1$ ainsi qu'une constante $\delta>0$ tels que
\begin{equation*}
\cN(F,x)= x\,P(\log{x})+O(x^{1-\delta})
\end{equation*}
lorsque $x\to \infty$. \qed
\end{theo}

\begin{rema}  \label{rema:rhoKexemples}
Nous verrons au cours de la d\'emonstration de la proposition~\ref{prop:XKVKcle} que la valeur de $\rho_K$ ne d\'epend que de $g$ et du groupe $\Gamma$; il est alors facile de la calculer dans des cas simples. Le tableau suivant donnent quelques exemples:---
\begin{equation*}
\begin{array}{c|| c | c | c | c | c | c | c}
g      &   1    &   2  &  2          &   2   &  \ell                       & \ell &  g   \\
\hline
\Gamma &  C_2   &  C_4& C_2\times C_2& D_4   &  C_{2\ell}                  & D_{2\ell} & 2^g.\Gamma_1 \\
\hline
\rho_K &  1     &  1  &    2         &   1   &  1+\frac{2^{\ell-1}-1}{\ell}  & 1+\frac{2^{\ell-1}-1}{\ell} & 1 \\  
\end{array}
\end{equation*}
\end{rema} 
Ici, $C_m$ d\'esigne un groupe cyclique d'ordre $m$, $D_{m}$ le groupe di\'edral d'ordre $2m$, $\ell$ un nombre premier impair, $\Gamma_1$ le groupe de Galois de la cl\^oture galoisienne de $K_0$ dans $L$, $2^g.\Gamma_1$ une extension de $\Gamma_1$ par le produit direct $C_2^g$ de $g$ copies de $C_2$. Rappelons que, quel que soit $K$, il existe $r\in \{1,2,\dots, g\}$ et une suite exacte de groupes $1\to C_2^r\to \Gamma\to \Gamma_1\to 1$ (voir le corollaire~\ref{cor:CMgalois}). 

Le tableau liste toutes les possibilit\'es lorsque $g\leq 2$. On peut \'egalement en d\'eduire toutes les possibilit\'es lorsque $g=3$ car on sait que dans ce cas $\Gamma$ est isomorphe \`a l'un des groupes $C_6$, $D_6$ ou $2^3.\Gamma_1$ avec $\Gamma_1=C_3$ ou $D_3$. Pour le calcul des valeurs des $\rho_K$ voir la remarque~\ref{rem:rhoKcalcul}.

\subsection{Travaux ant\'erieurs} \`A notre connaissance, les premiers travaux con\-sacr\'es \`a l'\'etude de la croissance asymptotique de la distribution des nombres de Weil dans un corps fix\'e sont ceux de Greaves et Odoni \cite{GrOdoni88} et Odoni \cite{Odoni91}. Dans \cite{GrOdoni88}, les auteurs consid\`erent les ensembles $\cW_K(n)$ o\`u $n$ est de la forme $p^k$, $k$ \'etant un entier fix\'e et $p$ variant sur l'ensemble des nombres premiers.  Le cas $k=1$ est \'egalement consid\'er\'e dans \cite{BoxGru15} Theorem~2.4. Par exemple, on tire de ces r\'esultats qu'il existe un nombre rationnel $d_K>0$ tel que le nombre $\cN_p(K,x):=\sum_{1\leq p\leq x, p \text{ premier}}a_K(p)$ est \'equivalent lorsque $x\to +\infty$ \`a $d_K\frac{x}{\log{x}}$; on remarquera que cette formule ne fait pas intervenir explicitement $\rho_K$ ce qui est en contraste avec le th\'eor\`eme~\ref{theo:premvers} du pr\'esent papier. 

Dans \cite{Odoni91}, Odoni \'etudie la croissance de la fonction $\cN_{\ve}(K,x)$ d\'efinie par la somme partielle $\sum_{1\leq n\leq x}\ve_K(n)$, o\`u $\ve_K(n)=1$ si $\cW_K(n)\neq \emptyset$ et $\ve_K(n)=0$ si $\cW_K(n)= \emptyset$. Dans ce but, il introduit \'egalement la fonction $\eta_K:\bN^*\to \{0,1\}$ d\'efinie par $\eta_K(n)=1$ ou $\eta_K(n)=0$ selon qu'il existe $\ga\in I^{w}_K$ tel que $\ga\overline{\ga}=n\cO_K$ ou non. La fonction $\eta_K$ \'etant multiplicative, Odoni est amen\'e \`a \'etudier la fonction z\^eta associ\'ee. Des groupes en relation avec notre groupe $\cC^{w}_K$ apparaissent \'egalement dans la \S~2 de son travail, qui contient \'egalement des id\'ees semblables \`a notre proposition~\ref{prop:gaPhigP}. Toutefois, le lien entre les formules qu'il obtient et la n\^otre reste \`a pr\'eciser. 

Comme nous l'avons d\'ej\`a signal\'e, la premi\`ere d\'emonstration compl\`ete de la conjecture de Manin dans le cadre des vari\'et\'es toriques a \'et\'e donn\'ee par Batyrev et Tschinkel~\cite{BatTsch95} et \cite{BatTsch98}. Depuis, d'autres d\'emonstrations ont \'et\'e trouv\'ees, notamment par Salberger~\cite{Sal98} puis par de la Bret\`eche~\cite{delaBreteche01}. Nous reviendrons dans le \subsecno~\ref{subsec:liensautres} aux liens pr\'ecis entre notre travail et ces d\'emonstrations, qui ne sont pas compl\`etement \'elucid\'es (au moins dans la t\^ete de l'auteur), faute d'interpr\'etation satisfaisante du groupe $\cC^{w}_K$ (ou plut\^ot d'un certain quotient de $\cC^{w}_K$ not\'e $\cS_K$) en termes des invariants arithm\'etiques usuels du tore $V^K$. Cela a comme cons\'equence que nous ne savons pas exprimer les fonctions $Z(K,\chi, s)$ comme int\'egrales ad\'eliques de mani\`ere satisfaisante.

Le probl\`eme de l'ensemble de m\'eromorphie des fonctions z\^eta des hauteurs a \'et\'e abord\'e par de nombreux auteurs: on trouvera un bilan des r\'esultats obtenus jusqu'en 2009 dans \cite{delaBreteche09} et le sujet continue de faire objet de nombreuses publications. Toutefois \`a la connaissance de l'auteur, les th\'eor\`emes~\ref{theo:prolmer} et \ref{theo:asymptpol} ne sont pas couverts par les r\'esultats parus dans la litt\'erature jusqu'ici. 

L'\'etude du comptage des points rationnels appartenant \`a des intersections (compl\`etes ou non) d'hypersurfaces projectives a d\'ebut\'ee avec le travail de Birch~\cite{Birch62}.  Malgr\'e des progr\`es r\'ecents (voir par exemple \cite{Die15}, \cite{BroHB17} et \cite{Schind15}), ces r\'esultats, dont les d\'emonstrations sont bas\'ees sur la m\'ethode du cercle, s'appliquent lorsque le nombre de variables cro\^{\i}t exponentiellement par rapport aux degr\'es, ce qui n'est pas le cas pour $X^K$. 

Rappelons que les tores $W^K$ se rencontrent dans la nature comme ex\-emples de sous-tores maximaux du groupe $\GSp_{2g}(\bQ)$ des similitudes symplectiques et jouent un r\^ole important dans la th\'eorie des vari\'et\'es ab\'eliennes \`a multiplication complexe.  De plus, les valeurs propres d'une matrice carr\'ee $M$ \`a coefficients entiers, de taille $(2g,2g)$ et telle que ${}^tMM=nI_{2g}$ sont des exemples de $n$-nombres de Weil. Notons enfin que les articles \cite{DiPHowe98}, \cite{Loxt72} et \cite{StanZah15}, parmi d'autres, s'int\'eressent aux diff\'erents aspects de la distribution de nombres de Weil. 

\subsection{Plan du papier} Il se r\'esume ainsi. La \S~\ref{sec:CMrappels} est consacr\'ee \`a quelques rappels sur corps CM, les nombres de Weil et les types CM. Dans la \S~\ref{sec:letore}, nous d\'ecrivons les tores $W^K$ et $V^K$; nous d\'ecrivons \'egalement la compactification $X^K$ de $V^K$ et nous \'etablisserons quelques unes de ses propri\'et\'es.  En particulier, nous d\'emontrarons la proposition~\ref{prop:compaqVK}. Dans la \S~\ref{sec:prodeul} nous \'etudions les produits eul\'eriens des fonctions $Z(K,\chi,s)$ et nous d\'emontrerons la th\'eor\`eme~\ref{theo:prolmer}, \`a l'exception de la derni\`ere phrase concernant le cas $g=2$. Le \subsecno~\ref{subsec:liensautres} est consacr\'e aux liens entre notre travail et ceux d'autres auteurs; nous y montrons notamment que les tores $V^K$ et $W^K$ satisfont \`a l'approximation faible. Enfin, dans la \S~\ref{sec:gendeux} nous \'etudierons le cas particulier o\`u $g=2$. 

\begin{remerc}
Ce travail a \'et\'e commenc\'e pendant un s\'ejour \`a l'Universit\'e de Colorado \`a Boulder et j'aimerais remercier cet \'etablissement et plus particuli\`erement David Grant pour l'organisation de ce s\'ejour. Je remercie \'egalement Felipe Voloch et Mirela Ciperiani pour m'avoir accueilli \`a l'Universit\'e de Texas \`a Austin pendant une visite courte mais tr\`es fructueuse. 
\end{remerc} 

\section{Rappels sur les corps CM et les nombres de Weil}\label{sec:CMrappels} 

Cette section est consacr\'ee aux rappels sur les corps CM, les nombres de Weil et les types CM. Les d\'emonstrations manquantes se trouvent dans \cite{GrOdoni88} et \cite{Shim97}.  

\subsection{Les corps CM} Soit $F$ un corps commutatif, qui est extension alg\'ebrique de $\bQ$. On note $\Hom(F,\bC)$ l'ensemble des plongements de $F$ dans $\bC$. Si $\phi\in \Hom(F,\bC)$, on note $\overline{\phi}$ l'\'el\'ement de $\Hom(F,\bC)$ d\'efini par $\overline{\phi}(\alpha)=\overline{\phi(\alpha)}$.

\begin{defi}\label{def:corpsCM}
 On dit que $F$ est \emph{totalement r\'eel} (resp. \emph{totalement imaginaire}) si, quel que soit le plongement $\phi\in \Hom(F,\bC)$, on a $\phi(F)\subseteq \bR$ (resp. $\phi(F)\not\subseteq \bR$).  On dit que $F$ est un \emph{corps CM} si s'il est totalement imaginaire et extension quadratique d'un sous-corps totalement r\'eel.  
\end{defi}

\begin{prop}
Soit $K$ un corps CM et soit $K_0$ un sous-corps totalement r\'eel de $K$ tel que $[K:K_0]=2$.

\case{i} Tout sous-corps de $K$ est soit totalement r\'eel soit un corps CM. 

\case{ii} Le sous-corps $K_0$ est uniquement d\'etermin\'e par $K$. Tout sous-corps totalement r\'eel de $K$ est contenu dans $K_0$. 

\case{iii} Notons $c$ l'automorphisme non trivial de $K$ fixant $K_0$. Alors $c$ commute avec tout autre automorphisme de $K/\bQ$. Si $\alpha\in K$, alors $\phi(c(\alpha))=\overline{\phi}(\alpha)$ pour tout $\phi\in \Hom(K,\bC)$.  \hfill \qed
\end{prop}

\begin{defi}
En vertu du point \case{ii}, on appelle $K_0$ le \emph{sous-corps r\'eel maximal} de $K$. En vertu du point \case{iii}, on appelle $c$ la \emph{conjugaison complexe} de $K$.
\end{defi}

Si $\alpha\in K$, on \'ecrit souvent $\overline{\alpha}$ pour $c(\alpha)$. Si $K_0$ est un corps totalement r\'eel et si $\alpha\in K_0$, on pose $\overline{\alpha}=\alpha$.

\begin{prop} Soit $F$ une extension alg\'ebrique de $\bQ$. Notons $\cE^{\tr}(F)$ (resp. $\cE^{\cm}(F)$) l'ensemble des sous-corps totalement r\'eels (resp CM) de $F$.

\case{i} Toute intersection de corps appartenant \`a  $\cE^{\tr}(F)$ appartient \`a $\cE^{\tr}(F)$. Tout corps engendr\'e par une famille de corps appartenant \`a $\cE^{\tr}(F)$ appartient \`a $\cE^{\tr}(F)$. 

\case{ii} Toute intersection de corps appartenant \`a  $\cE^{\cm}(F)$ appartient \`a $\cE^{\tr}(F)$ ou \`a $\cE^{\cm}(F)$. Tout corps engendr\'e par une famille de corps appartenant \`a $\cE^{\cm}(F)$ appartient \`a $\cE^{\cm}(F)$.  \hfill \qed
\end{prop}

\begin{coro}\label{coro:maxtrcm}  Soit $F$ une extension alg\'ebrique de $\bQ$. 

\case{i} $F$ poss\`ede un unique sous-corps totalement r\'eel maximal (pour l'in\-clusion). Si $F$ contient au moins un sous-corps CM, alors il contient un unique sous-corps CM maximal.

\case{ii} On suppose $F$ galoisienne sur $\bQ$. Si $K_0\in \cE^{\tr}(F)$, alors la cl\^oture galoisienne de $K_0$ dans $F$ appartient \`a $\cE^{\tr}(F)$. De m\^eme, si $K\in \cE^{\cm}(F)$, alors la cl\^oture galoisienne de $K$ dans $F$ appartient \`a $\cE^{\cm}(F)$.\hfill \qed
\end{coro}

\begin{coro} \label{cor:CMgalois}
Soit $K$ un corps CM et soit $L$ une cl\^oture galoisienne de $K$ sur $\bQ$. Soit $K_0$ (resp.$L_0$) le sous-corps r\'eel maximal de $K$ (resp. $L$), $\Gamma$ le groupe de Galois de $L/\bQ$ et $c$ le g\'en\'erateur de $\Gamma^{L_0}$. Alors $L_0$ est galoisienne sur $\bQ$, $c$ est un \'el\'ement central de $\Gamma$ et, si on note $\Gamma_0$ (resp. $\Gamma_1$) le groupe de Galois de $L_0/\bQ$ (resp. de la cl\^oture galoisienne de $K_0$ dans $L$ sur $\bQ$), il existe un entier $r\in \{1,2,\dots, g\}$ et un diagramme commutatif \`a lignes horizontales exactes
\begin{equation*}
\begin{tikzcd}
\hskip16mm & 1 \arrow[r] &<c> \arrow[r] \arrow[d, hook, to head] &\Gamma \arrow[r] \arrow[d, equal] &\Gamma_0 \arrow[r] \arrow[d, two heads] &1\\
\hskip16mm & 1 \arrow[r] &C_2^r \arrow[r]      & \Gamma \arrow[r]  & \Gamma_1 \arrow[r] & 1 & \hskip17mm \qed
\end{tikzcd}
\end{equation*}
\end{coro}

\begin{lemm} \label{lem:primKQ} Soit $K$ un corps CM. Alors il existe $\ve\in K$ tel que $K=\bQ(\ve)$ et $\ve^2\in K_0$.
\end{lemm}

\begin{proof}
Posons $2g=[K:\bQ]$. Puisque $K$ est une extension quadratique de $K_0$, il existe $\ve_0\in K$ tel que $\ve_0^2\in K_0$ et $K=K_0(\ve)$. Il est clair que $c(\ve_0)=-\ve_0$. il suit que le $K_0$-espace vectoriel $V:=\{\ve\in K\mid c(\ve)=-\ve\}$ est de dimension un, avec base $\ve_0$. Supposons que $V$ ne contient aucun \'el\'ement $\ve$ tel que $K=\bQ(\ve)$. Alors pour tout $\ve\in V$, le degr\'e de l'extension $\bQ(\ve)/\bQ$ divise $2g$, et n'est pas \'egal \`a $2g$. Elle est donc de degr\'e au plus $g$. Or, $V$ est un $\bQ$-espace vectoriel de dimension $g$: en utilisant le fait que $K$ ne contient qu'un nombre fini de sous-corps et le fait qu'un $\bQ$-espace vectoriel non-nul n'est pas une r\'eunion finie de sous-espaces propres, on voit que $V$ est \'egal \`a un sous-corps de $K$ de degr\'e $g$. Puisque $V$ n'est pas stable par multiplication, on aboutit \`a une contradiction.
\end{proof}

\subsection{Les nombres de Weil} \label{subsec:nWeil} Si $F$ est une extension alg\'ebrique sur $\bQ$, on note $\cO_F$ son anneau des entiers (c'est-\`a-dire la cl\^oture int\'egrale de $\bZ$ dans $F$). 
\begin{defi}\label{def:noWeil}
Soit $\alpha$ un \'el\'ement d'une extension alg\'ebrique $F$ de $\bQ$ et soit $n\in \bN^*$. On dit que $\alpha$ est un \emph{$n$-nombre de Weil} si $\alpha\in \cO_F$ et si  pour tout $\phi\in \Hom(F,\bC)$ on a $\phi(\alpha)\overline{\phi}(\alpha)\in \bN^*$. On dit que $\alpha$ est un \emph{nombre de Weil} (au sens g\'en\'eral) s'il existe $n$ tel que $\alpha$ soit un $n$-nombre de Weil.
\end{defi}

Comme dans l'introduction, on note $\cW_F(n)$ l'ensemble des $n$-nombres de Weil appartenant \`a $F$ et on pose $\cW_F=\bigcup_{n\geq 1}\cW_F(n)$. 

\begin{enonce}
(1) En utilisant le fait que tout plongement dans $\bC$ d'un sous-corps de $F$ s'\'etend \`a un plongement de $F$ on voit facilement que cette d\'efinition ne d\'epend pas du choix de $F$ mais uniqument du sous-corps $\bQ(\alpha)$ engendr\'e par $\alpha$.  

(2) Nous avons ajout\'e l'\'epith\`ete {\og}au sens g\'en\'eral{\fg} car un bon nombre d'auteurs r\'eserve le terme \emph{nombre de Weil} aux nombres alg\'ebriques entiers sur $\bZ$ et tels que $\phi(\alpha)\overline{\phi}(\alpha)$ est une puissance d'un nombre premier. Toutefois, dans ce papier, la notion de nombre de Weil sera toujours entendue dans le sens de la d\'efinition~\ref{def:noWeil}. Par ailleurs, dans \cite{GrOdoni88} et \cite{Odoni91}, les $\alpha$ tels que $\alpha^2\in \bQ_+$ ne sont pas consid\'er\'es comme \'etant des nombres de Weil. 
\end{enonce}

\begin{prop}\label{prop:WReouCM}
Soit $\alpha\in F$. On suppose qu'il existe $q\in \bQ$ tel que $\phi(\alpha)\overline{\phi}(\alpha)=q$ pour tout $\phi\in \Hom(F,\bC)$. 

\case{i} Quel que soit $\phi\in \Hom(F,\bC)$, $|\phi(\alpha)|=\sqrt{q}$.

\case{ii} Ou bien $\alpha^2=q$ auquel cas $\bQ(\alpha)$ est r\'eel, ou bien $\bQ(\alpha)$ est un corps CM. Si $\bQ(\alpha)$ est un corps CM, alors son sous-corps r\'eel maximal est $\bQ(\alpha+\frac{q}{\alpha})$. 
\end{prop}

\begin{proof}
\case{i} Puisque $|\phi(\alpha)|=|\overline{\phi(\alpha)}|$, c'est clair.

\case{ii} Notons $M_\alpha\in \bQ[x]$ le polyn\^ome minimal de $\alpha$. On consid\`ere deux cas.

Si $M_{\alpha}$ poss\`ede une racine r\'eelle, alors il existe $\phi\in \Hom(\bQ(\alpha),\bC)$ qui se factorise par l'inclusion $\bR\subseteq \bC$. Alors $0\leq q=\phi(\alpha)\overline{\phi(\alpha)}=\phi(\alpha)^2$, d'o\`u $\phi(\alpha^2)=q$ ce qui implique que $M_\alpha$ est de degr\'e $1$ ou $2$ et que $\alpha^2=q$. Comme $q>0$, $\bQ(\alpha)$ est soit \'egal \`a $\bQ$ soit un corps quadratique r\'eel.

Si $M_{\alpha}$ ne poss\`ede aucune racine r\'eelle, alors le corps $\bQ(\alpha)$ est totalement imaginaire. En outre, $q>0$ et $\alpha\neq 0$. Posons $\lambda=\alpha+\frac{q}{\alpha}$. Si $\phi\in \Hom(F,\bC)$, alors $\phi(\lambda)=\phi(\alpha)+\phi(\frac{q}{\alpha})=\phi(\alpha)+\overline{\phi}(\alpha)\in \bR$. On en tire que $\bQ(\lambda)$ est totalement r\'eel. D'autre part, $\bQ(\alpha)$ est une extension quadratique de $\bQ(\lambda)$, car $\alpha$ est racine de $x^2-\lambda x+q$ et, $\bQ(\alpha)$ \'etant totalement imaginaire et $\bQ(\lambda)$ \'etant totalement r\'eel, l'\'eventualit\'e $\bQ(\alpha)=\bQ(\lambda)$ est exclue.\end{proof}

\begin{lemm} \label{lemm:cnsKw} Soit $K$ un corps CM et soit $\alpha\in K$. Alors les propri\'et\'es suivantes sont \'equivalentes.

\case{i} $\alpha\in K^{w,\times}\cup \{0\}$.

\case{ii} $\phi(\alpha)\overline{\phi}(\alpha)$ est ind\'ependent du choix de $\phi\in \Hom(K,\bC)$.
\end{lemm}

\begin{proof} L'implication \case{i}$\Longrightarrow$\case{ii} est claire. Montrons la r\'eciproque. Posons $q=\phi(\alpha)\overline{\phi}(\alpha)$. Si $L$ d\'esigne la cl\^oture alg\'ebrique de $K$ dans $\bC$, alors $q\in L$. Rappelons que $\Gamma$ d\'esigne le groupe de Galois de $L$ sur $\bQ$. Puisque $c$ est un \'el\'ement central de $\Gamma$, on a $\gamma\circ \overline{\phi}=\overline{\gamma\circ \phi}$. L'hypoth\`ese \case{ii} implique alors que $\gamma(q)=q$ pour tout $\gamma\in \Gamma$, d'o\`u $q\in \bQ$.  
\end{proof}

\begin{lemm}\label{lem:racinedeun}
Soient $\alpha$, $\beta$ deux nombres de Weil appartenant \`a l'extension alg\'ebrique $F$ de $\bQ$. Si $\alpha\cO_F=\beta\cO_F$, alors il existe une racine de l'unit\'e $\zeta\in F$ tel que $\beta=\zeta\alpha$. 
\end{lemm}

\begin{proof}
On peut supposer que $F$ est une extension finie de $\bQ$, qui est soit un corps quadratique r\'eel soit un corps CM. Puisque $\alpha$ et $\beta$ engendrent le m\^eme id\'eal de $\cO_F$, il existe $\zeta \in \cO_{F}^\times$ tel que $\beta=\zeta\alpha$. D'autre part, puisque $\alpha\overline{\alpha}$, $\beta\overline{\beta}\in \bN^*$ et $\alpha\overline{\alpha}\cO_F=\beta\overline{\beta}\cO_F$, on voit que $\alpha\overline{\alpha}=\beta\overline{\beta}$. Notons cet entier naturel $n$. Alors $|\phi(\alpha)|=|\phi(\beta)|=\sqrt{n}$ pour tout $\phi\in \Hom(F,\bC)$, d'o\`u $|\phi(\zeta)|=1$ pour tout $\phi\in \Hom(F,\bC)$. Un th\'eor\`eme bien connu de Kronecker implique alors que $\zeta$ est une racine de l'unit\'e.  \end{proof}

\begin{theo}\label{theo:wFnfini}
Soit $F$ une extension finie de $\bQ$ et soit $n\in \bN^*$. 

\case{a} L'ensemble $\cW_F(n)$ est fini. 

\case{b} Si $K$ d\'esigne le sous-corps CM maximal de $F$, alors $a_K(n)\leq a_F(n)\leq a_K(n)+2$ et, lorsque $x\to \infty$, $\cN(F,x)-\cN(K,x)=O(\sqrt{x})$. 
\end{theo}

\begin{proof}
\case{a} Puisque $F$ ne contient qu'un nombre fini de sous-corps, il suffit de traiter le cas o\`u $F$ est un corps quadratique r\'eel ou un corps CM. Soit $\alpha\in \cW_F(n)$. Puisque $\cO_F$ est un anneau de Dedekind, l'id\'eal $n\cO_F$ n'est divisible que par un nombre fini d'id\'eaux de $\cO_F$. En particulier, comme $\alpha\overline{\alpha}\cO_F=n\cO_F$, il n'y a qu'un nombre fini de possibilit\'es pour l'id\'eal $\alpha\cO_K$. On conclut en \'evoquant le lemme~\ref{lem:racinedeun}

\case{b} Puisque $K\subseteq F$, on a $a_K(n)\leq a_F(n)$. D'apr\`es le point \case{ii} de la proposition~\ref{prop:WReouCM}, un $n$-nombre de Weil $\alpha$ n'appartenant pas \`a $K$ satisfait \`a $\alpha^2=n$, d'o\`u $a_F(n)\leq a_K(n)+2$. D'autre part, $F$ ne contenant qu'un nombre fini de sous-corps quadratiques, on voit que l'ensemble $\cS_F$ des entiers $n_0>1$ sans facteur carr\'e tels que $\sqrt{n_0}\in F$ est fini. Par cons\'equent, s'il existe $\alpha\in F$ tel que $\alpha\notin K$ et $\alpha^2=n$, alors $n=m^2n_0$ avec $m\in \bN^*$ et $n_0\in \cS_F$.  D'o\`u l'estimation $\cN(F,x)-\cN(K,x)=O(\sqrt{x})$. 
\end{proof}

\subsection{Types CM}\label{subsection:typesCM} Fixons un corps CM $K$ de degr\'e fini sur $\bQ$;  on note $K_0$ son sous-corps r\'eel maximal, $L$ la cl\^oture galoisienne de $K$ dans $\bC$, $\Gamma$ le groupe de Galois de $L$ sur $\bQ$. Rappelons que $L$ est un corps CM: on note $L_0$ son sous-corps r\'eel maximal. On pose $g=[K_0:\bQ]$, $g_L=[L_0:\bQ]$. Si $F$ est un sous-corps de $L$, on note $\Gamma^F$ le sous-groupe de $\Gamma$ correspondant \`a $F$. 

Un \emph{type CM} sur $K$ est un sous-ensemble $\Phi$ de $\Hom(K,\bC)$  tel que $\Phi\cup \overline{\Phi}=\Hom(K,\bC)$ et $\Phi\cap \overline{\Phi}=\emptyset$. Ici, $\overline{\Phi}$ d\'esigne l'ensemble $\overline{\phi}\mid \phi\in \Phi\}$. Si $\Phi$ est un type CM, alors $\Hom(K_0,\bC)=\{\phi|_{K_0}\mid \phi\in \Phi\}$ et, r\'eciproquement, tout type CM sur $K$ est obtenu en choisissant, pour tout $\phi_0\in \Hom(K_0,\bC)$, une extension de $\phi_0$ \`a $K$. Ainsi, il y a $2^g$ types CM sur $K$. On d\'esigne l'ensemble de ces types CM par $\cC\cM_K$. On trouvera davantage d'information sur les types CM dans le chapitre 2 de \cite{Shim97}; par la suite nous nous limiterons donc \`a des rappels brefs.

Si $\phi\in \Hom(K,\bC)$, alors $\phi(K)\subseteq L$ et on peut donc consid\'erer les \'el\'ements de $\Hom(K,\bC)$ comme des \'el\'ements de $\Hom(K,L)$. En outre, $L$ \'etant galoisien sur $\bQ$, tout \'el\'ement de  $\Hom(K,L)$ se prolonge en un \'el\'ement de $\Gamma$. Deux \'el\'ements $\sigma$, $\tau\in \Gamma$ d\'efinissent le m\^eme \'el\'ement de $\Hom(K,L)$ si et seulement s'il existe $\gamma\in \Gamma^K$ tel que $\tau=\sigma\circ \gamma$. 

Le groupe $\Gamma$ op\`ere \`a gauche sur $\cC\cM_K$: si $\gamma\in \Gamma$ et si $\Phi\in \cC\cM_K$, alors $\gamma\cdot\Phi=\{\gamma\circ\phi\mid \phi\in \Phi\}$. 

Soit $\Phi\in \cC\cM_K$ et soit $S\subseteq \Gamma$ l'ensemble des automorphismes de $L$ prolongeant les \'el\'ements de $\Phi$. Alors $S$ est un type CM sur $L$ et le groupe $\{\gamma\in \Gamma\mid S\gamma=S\}$ contient $\Gamma^K$. R\'eciproquement, si $S$ est un type CM sur $L$ tel que $\{\gamma\in \Gamma\mid S\gamma=S\}$ contienne $\Gamma^K$, alors $S$ est le l'ensemble des prolongements d'un type CM sur $K$.

D'autre part, si $\Phi\in \cC\cM_K$, alors le sous-groupe $H=\{\gamma\in \Gamma\mid \gamma S=S\}$ ne contient pas la conjugaison complexe et le sous-corps $\hat{K}=\hat{K}_\Phi$ de $L$ correspondant \`a $H$ est un corps CM, appel\'e le \emph{corps reflex} de $(K,\Phi)$. On voit facilement que si $\Phi$, $\Phi'\in \cC\cM_K$ sont dans la m\^eme orbite sous $\Gamma$, les corps $\hat{K}_\Phi$ et $\hat{K}_{\Phi'}$ sont conjugu\'es sous $\Gamma$. 

\begin{prop}\label{prop:OrbPhiHomhatKLiso} Soit $\Phi\in \cC\cM_K$ et soit $\Orb{(\Phi)}$ l'orbite de $\Phi$ sous $\Gamma$. Alors les $\Gamma$-ensembles $\Orb(\Phi)$ et $\Hom(\hat{K}_\Phi,L)$ sont isomorphes.
\end{prop}

\begin{proof} Soit $\iota$ l'inclusion de $\hat{K}_\Phi$ dans $L$. Alors le stabilisateur de $\iota$ est \'egal \`a $\Gamma^{\hat{K}}$, ce qui est \'egalement le stabilisateur de $\Phi$. Les actions de $\Gamma$ sur les ensembles finis $\Orb{(\Phi)}$ et $\Hom(\hat{K},\bC)$ \'etant transitives, on peut d\'efinir une bijection $u:\Hom(\hat{K},\bC)\to \Orb{(\Phi)}$ par la r\`egle
\begin{equation*}
u(\gamma\circ \iota)= \gamma\cdot \Phi,  \qquad (\gamma\in \Gamma).
\end{equation*}
et on v\'erifie sans peine qu'il s'agit d'un $\Gamma$-morphisme.
\end{proof}

Soit $\Phi\in \cC\cM_K$. On d\'efinit un type CM sur $\hat{K}_\Phi$, appel\'e \emph{le reflex} de $\Phi$ et not\'e $\hat{\Phi}$, de la mani\`ere suivante. Soit $S$ l'ensemble de tous les prolongements des \'el\'ements de $\Phi$ \`a des automorphismes de $L$. Alors $S$ et $S^*=\{\sigma^{-1}\mid \sigma\in S\}$ sont des types CM sur $L$ et on v\'erifie que $S^*$ est l'ensemble des prolongements \`a $L$ d'un type CM sur $\hat{K}_\Phi$ qui sera not\'e $\hat{\Phi}$. On d\'efinit alors la \emph{norme} associ\'ee \`a $\hat{\Phi}$ comme \'etant le morphisme de groupes $\N_{\hat{\Phi}}:\hat{K}^\times\to L^\times$ donn\'ee par
$\N_{\hat{\Phi}}(\alpha)=\prod_{\psi\in \hat{\Phi}}\psi(\alpha)$.
Le lemme suivant est bien connu:

\begin{lemm}\label{lemm:normPhihat} La norme $\N_{\hat{\Phi}}$ associ\'ee \`a $\hat{\Phi}$ prennent ses valeurs dans $K^\times$ et, en fait, dans le groupe $K^{w,\times}=\{\alpha\in K^\times\mid \alpha\overline{\alpha}\in \bQ^\times\}$. \qed
\end{lemm}

On en tire facilement que $\N_{\hat{\Phi}}$ induit alors des morphismes de groupes (encore not\'es $\N_{\hat{\Phi}}$) $I_{\hat{K}}\to I^w_K$, $P_{\hat{K}}\to P^w_K$, $\Pic(\cO_{\hat{K}})\to \cC^w_K$ qui seront utilis\'es par la suite. 

\section{Les tores $W^K$ et $V^K$ et la compactification $X^K$}\label{sec:letore}

\subsection{Notations} On reprend les notations introduites au d\'ebut du \subsecno~\ref{subsection:typesCM}. Si $F$ est une extension finie de $\bQ$, on note $T^F$ la restriction de Weil du groupe multiplicatif $\bG_{m,F}$ \`a $\bQ$. Rappelons que $T^F$ est un tore sur $\bQ$ dimension $[F:\bQ]$ tel que, pour toute $\bQ$-alg\`ebre $E$, on ait $T^K(E)\simeq \bG_{m,F}(E\otimes_{\bQ} F)$ fonctoriellement en $E$. 

\subsection{D\'efinition des tores $W^K$ et $V^K$} La norme $\N_{K/K_0}$ de $K$ sur $K_0$ induit un morphisme de $\bQ$-tores $T^K\to T^{K_0}$ (\'egalement not\'e $\N_{K/K_0}$); d\'esignons par $T^{K/K_0,1}$ son noyau. D'autre part, l'inclusion $\bQ^\times\subseteq  K_0^\times$ induit une inclusion de $\bG_{m,\bQ}$ dans $T^{K_0}$. Le tore $W^K$ est alors d\'efini comme \'etant l'image r\'eciproque par $\N_{K/K_0}$ de $\bG_{m,\bQ}$; il s'agit donc d'un sous-tore de $T^K$. La sitution est r\'esum\'ee par le diagramme commutatif \`a lignes exactes
\begin{equation*}
\begin{tikzcd}
1 \arrow[r] &T^{K/K_0,1} \arrow[r]                &T^K \arrow[r, "\N_{K/K_0}"]              &T^{K_0} \arrow[r]                          &1\\
1 \arrow[r] &T^{K/K_0,1} \arrow[r]\arrow[u,equal] &W^K \arrow[r] \arrow[u, hook, to head]   &\bG_{m,\bQ} \arrow[r] \arrow[u, hook, to head] &1. 
\end{tikzcd}
\end{equation*}

D'autre part, l'inclusion $\bQ^\times\subseteq  K^\times$ induit une inclusion de $\bG_{m,\bQ}$ dans $T^{K}$ qui se factorise par $W^K$. Le tore $V^K$ est alors d\'efini par la suite exacte
\begin{equation*}
\begin{tikzcd}
1\arrow[r] &\bG_{m,\bQ}\arrow[r] &  W^K\arrow[r] & V^K\arrow[r]  & 1. &  
\end{tikzcd}\tag*{$(*)$}
\end{equation*}

Si $E$ est une $\bQ$-alg\`ebre, la conjugaison complexe $\alpha\mapsto \overline{\alpha}$ s'\'etend \`a $E\otimes_{\bQ} K$ par $\overline{u\otimes \alpha}\mapsto u\otimes \overline{\alpha}$. Si $\alpha\in K^\times$, alors $\N_{K/K_0}(\alpha)=\alpha\overline{\alpha}$ et $\N_{K/K_0}$ s'\'etend naturellement en un morphisme $(E\otimes_{\bQ} K)^\times\to (E\otimes_{\bQ} K_0)^\times$ que nous d\'esignons encore par $\N_{K/K_0}$.

Le corps $K$ \'etant fix\'e, on supprime d\'esormais l'indice $K$ de la notation si aucune confusion n'est \`a craindre. 

\begin{lemm} \label{lem:WVprop}
\case{i} On a $\dim{W}=g+1$ et $\dim{V}=g$.

\case{ii} La compos\'ee de l'inclusion $T^{K/K_0,1}\subseteq W$ et de la fl\`eche $W\to V$ est une isog\'enie de $T^{K/K_0,1}$ sur $V$.

\case{iii} $W$ et $V$ sont d\'eploy\'es par $L$. 

\case{iv} Pour toute $\bQ$-alg\`ebre $E$, on a $W(E)\simeq \{\alpha\in (E\otimes_{\bQ} K)^\times \mid \N_{K/K_0}(\alpha)\in E^\times\}$ fonctoriellement en $E$ et, si $\Gamma$ op\`ere trivialement sur $E$, $V(E)\simeq W(E)/E^\times$ fonctoriellement en $E$.  
\end{lemm}

\begin{proof} Le point \case{i} d\'ecoule de la discussion pr\'ec\'edente.

\case{ii} Puisque $\dim{T^{K/K_0,1}}=\dim{V}$, il suffit de montrer que le noyau est fini. C'est clair car $\ker(W\to V)=\bG_{m,\bQ}$ et $\bG_{m,\bQ}\not\subseteq T^{K/K_0,1}$. 

\case{iii} On sait que $L$ d\'eploie $T^K$. Puisque tout sous-tore d'un tore d\'eploy\'e sur un corps de caract\'eristique nulle est d\'eploy\'e, on en tire que $W$ est d\'eploy\'e par $L$. Enfin $V$ est d\'eploy\'e par $L$ en raison de la suite exacte $(*)$.  

\case{iv} La description de $W(E)$ d\'ecoule de la d\'efinition de $W$, puis celle de $V(E)$ du th\'eor\`eme $90$ de Hilbert. 
\end{proof}

\subsection{Le groupe des caract\`eres} En g\'en\'eral, si $T$ est un tore sur $\bQ$, on note $\hat{T}$ son groupe de caract\`eres $\Hom(T,\bG_{m,\bQ})$. Dans tous les cas qui nous concernent, $T$ sera d\'eploy\'e par $L$ et nous supposerons par la suite que c'est le cas. Ainsi $\hat{T}$ poss\`ede alors une structure naturelle de $\Gamma$-module. 

Posons $\Xi=\Hom(K,\bC)$, $\Xi_0=\Hom(K_0,\bC)$. Alors $\hat{T}^K$ s'identifie avec le groupe ab\'elien libre $\bZ[\Xi]$, l'\'el\'ement $\sum_{\phi\in \Xi}n_\phi\phi$ correspondant au caract\`ere $\alpha\mapsto \prod_{\phi\in \Xi}\phi(\alpha)^{n_\phi}$ (o\`u $\alpha\in K^\times$). L'action de $\Gamma$ est alors donn\'e par $(\gamma,\phi)\mapsto \gamma\circ\phi$. De la m\^eme mani\`ere, le $\Gamma_0$-module $\hat{T}^{K_0}$ s'identifie avec $\bZ[\Xi_0]$, que l'on consid\`ere comme un $\Gamma$-module \emph{via} l'application quotient $\Gamma\to \Gamma_0$. Notons $I_\Xi$ l'id\'eal augmentation de $\bZ[\Xi]$ et $J_\Xi$ le sous-module de $\bZ[\Xi]$ engendr\'e par les \'el\'ements $\phi+\overline{\phi}$ avec $\phi$ parcourant $\Xi$. Le lemme suivant est alors clair.

\begin{lemm}\label{lem:conjcomp} 
La conjugaison complexe op\`ere par $-1$ sur $\hat{T}^{K/K_0,1}$ et sur $\hat{V}$. \qed
\end{lemm}

\begin{lemm}\label{lem:isohats} 
On dispose d'isomorphismes naturels de $\Gamma$-modules $\hat{T}^{K/K_0,1}\simeq \bZ[\Xi]/J_\Xi$, 
$\hat{W}\simeq \bZ[\Xi]/(I_\Xi\cap J_\Xi)$ et $\hat{V}\simeq I_\Xi/(I_\Xi\cap J_\Xi)$
\end{lemm}

\begin{proof} Rappelons que, \`a tout morphisme de $\bQ$-tores $T\to S$ d\'eploy\'es par $L$ correspond un morphisme de $\Gamma$-modules $\hat{S}\to \hat{T}$, et que le foncteur $T\mapsto \Hom(T,\bG_{m,\bQ})$ transforme suites exactes en suites exactes. Ainsi, \`a la suite exacte de tores $1\to T^{K/K_0,1}\rightarrow T^K\xrightarrow{\N_{K/K_0}} T^{K_0}\rightarrow 1$ correspond une suite exacte de $\Gamma$-modules $0\to \bZ[\Xi_0]\to \bZ[\Xi]\to \hat{T}^{K/K_0,1}\to 0$. Or l'application $\bZ[\Xi_0]\to \bZ[\Xi]$ transforme $\phi_0$ en $\phi+\overline{\phi}$, $\phi$, $\overline{\phi}$ \'etant les deux prolongements de $\phi_0$ \`a $K$. Par cons\'equent, l'image de $\bZ[\Xi_0]$ and $\bZ[\Xi]$ est \'egale \`a $J_\Xi$, et on tire l'isomorphisme naturel $\hat{T}^{K/K_0,1}\simeq \bZ[\Xi]/J_\Xi$. On d\'emontre les deux autres de la m\^eme mani\`ere, en remarquant que $I_\Xi$ est le noyau du morphisme $\bZ[\Xi]\to \hat{\bG}_{m,\bQ}$ correspondant \`a l'inclusion $\bG_{m,\bQ}\hookrightarrow T^K$.\end{proof} 

\subsection{Un \'eventail complet et r\'egulier}\label{subsec:event} Posons $M=\hat{V}$, $N=\Hom(M,\bZ)$, que l'on consid\`ere comme $\Gamma$-modules. Si $A$ est un $\bZ$-module, on note $M_A$ (resp. $N_A$) le $A$-module $M\otimes A$ (resp. $N\otimes A$). Fixons une fois pour toutes un type CM $\Phi=\{\phi_1,\phi_2,\dots, \phi_g\}$ de $K$. On note $\cU$ le produit cart\'esien de $g$ copies de $\{-1,1\}$ index\'e par les \'el\'ements de $\Phi$.

Pour tout $\phi\in \Phi$, on note $f_\phi$ la fonction $\bZ$-lin\'eaire $\bZ[\Xi]\to \bZ$ d\'efinie sur les \'el\'ements $\psi$ de $\Xi$ par
\begin{equation*}
f_\phi(\psi)=\begin{cases}1 \text{ si $\psi=\phi$}\\
-1 \text{ si $\psi=\overline{\phi}$}\\
0 \text{ si $\psi\notin \{\phi, \overline{\phi}\}$}.
\end{cases}
\end{equation*}
Il est clair que les fonctions $f_\phi$ s'annulent sur $J_\Xi$, de sorte qu'on peut les consid\'erer comme des fonctions sur $\hat{T}^{K/K_0,1}$ et sur $M$ et donc, par $\bZ$-lin\'earit\'e, comme des \'el\'ements de $N$. On note \'egalement $f_\phi$ son image dans $N_\bR$ par l'application naturelle  $N\to N_\bR$. Il est facile de voir que $(f_\phi)_{\phi\in \Phi}$ est une base du $\bR$-espace vectoriel $N_\bR$.

Rappelons \cite{CoxLitSch11} (voir \cite{BatTsch98} pour des rappels brefs) qu'un \emph{\'eventail} dans $N_\bR$ est un ensemble $\Sigma$ de c\^ones poly\'edraux $\sigma$ tel que 

\case{i} Tout $\sigma\in \Sigma$ contient $0$;

\case{ii} Toute face d'un c\^one appartenant \`a $\Sigma$ appartient \`a $\Sigma$;

\case{iii} L'intersection de deux c\^ones de $\Sigma$ est une face de chacun d'eux.

De plus, l'\'eventail $\Sigma$ est dit \emph{complet} (resp. \emph{r\'egulier} ou \emph{lisse}) s'il poss\`ede la propri\'et\'e \case{iv} (resp. \case{v}) suivante:

\case{iv} $N_\bR=\bigcup_{\sigma\in \Sigma}\sigma$;

\case{v} Tout c\^one de $\Sigma$ est engendr\'e par un sous-ensemble d'une $\bZ$-base de $N$.

Pour chacun des $2^g$ \'el\'ements $u=(u_\phi)_{\phi\in \Phi}$ de $\cU$, on note $e_u$ l'\'el\'ement de $N_\bR$ d\'efinie par 
\begin{equation*}
e_u=\frac{1}{2}\sum_{\phi\in \Phi}u_\phi f_\phi.
\end{equation*}

Disons que deux \'el\'ements de $\cU$ sont \emph{voisins} s'ils diff\`erent en une et une seule coordonn\'ee. On d\'efinit alors un \'eventail $\Sigma$ dans $N_\bR$ de la mani\`ere suivante. Pour tout $d\in \{1,2,\dots, g\}$, l'ensemble $\Sigma(d)$ des c\^ones de dimension $d$ appartenant \`a $\Sigma$ est donn\'e par
\begin{align*}
\Sigma(d)=\big\{\sum_{s=1}^d\bR_+e_{u^{s}}\mid (u^{1},u^{2},\dots , u^{d})&\in \cU^r, u^{2}, \dots, u^{d} \hskip2mm\text{ sont des voisins}\\ 
&\text{ deux-\`a-deux distincts de $u^{1}$}\big\}
\end{align*} 
En particulier, $\Sigma(1)=\{\bR_+e_u\mid u\in \cU\}$. On v\'erifie que $\Sigma$ est bien un \'eventail complet et r\'egulier. (Afin de v\'erifier la propri\'et\'e \case{v}, il suffit de montrer que si $(u^{1},u^{2}, \dots, u^{g})\in \cU^g$ et si $u^{2}$, \dots, $u^{g}$ sont des voisins deux-\`a-deux distincts de $u^{1}$, alors $(u^{i})_{1\leq i\leq g}$ est une $\bZ$-base de $N$. Pour cela, on commence en remarquant que, si $\phi_0\in \Phi$ est fix\'e, alors $(\phi-\overline{\phi}_0+I_\Xi\cap J_\Xi)_{\phi\in \Phi}$ est une $\bZ$-base de $I_\Xi/(I_\Xi\cap J_\Xi)$ et que $e_u(\phi-\overline{\phi}_0)=\frac{1}{2}(u_\phi+u_{\phi_0})$. Un calcul court permet alors de conclure.) 

Selon la th\'eorie g\'en\'erale des vari\'et\'es toriques (voir par exemple \cite{CoxLitSch11}, o\`u la plupart du temps le corps de base est $\bC$ mais les r\'esultats dont nous aurons besoin sont de nature alg\'ebrique et s'\'etendent sans difficult\'e vari\'et\'es toriques sur un corps quelconque de caract\'eristique nulle et dont le tore dense est d\'eploy\'e), on peut associer \`a l'\'eventail $\Sigma$ une compactification \'equivariante et lisse de $V$ sur le corps $L$ d\'eployant $V$. Posons $\overline{V}=V\times_{\Spec{\bQ}}\Spec{L}$ et notons cette compactification $\overline{X}$.  Posons $\overline{\cD}=\overline{X}- \overline{V}$. La construction explicite de $\overline{V}$ et de $\overline{X}$ \`a partir de $\Sigma$ (voir \cite{CoxLitSch11} et qui est \'egalement rappel\'ee dans les Propositions~1.2.3 et 1.2.4 de \cite{BatTsch95}) fournit la description suivante de $\overline{X}$. 

\begin{theo} Le vari\'et\'e $\overline{X}$ est projective, isomorphe sur $L$ \`a la sous-vari\'et\'e de $\bP^{2g-1}$ d\'efinie par les \'equations $X_1X_2-X_{2i-1}X_{2i}=0$ (\/$2\leq i\leq g$). Sous cet isomorphisme, l'ouvert $\overline{V}$ de $\overline{X}$ correspond \`a l'ouvert $X_1X_2\neq 0$ et $\overline{\cD}$ est d\'efinie par les \'equations
\begin{equation*}X_1X_2=X_3X_4=\cdots =X_{2g-1}X_{2g}=0.  \tag*{$(*)$}
\end{equation*} 
Par cons\'equent, $\overline{\cD}$ est la r\'eunion d'un ensemble $\overline{\cP}$ de $2^g$ sous-espaces projectifs de dimension $g-1$. 

En plus, $-\overline{\cD}$ est un diviseur canonique de $\overline{X}$. L'espace de sections globales de $\cO_{\overline{X}}(\overline{\cD})$ est engendr\'e par $(X_jX_k)/(X_1X_2)$, $1\leq j\leq k\leq 2g$, avec pour seules relations celles provenant de $(*)$.  
\end{theo}

\begin{proof}
Il suffit de montrer que l'\'eventail d\'ecrit plus haut correspond bien \`a la vari\'et\'e $X$
 d\'efinie dans l'\'enonc\'e. Il est clair que $\overline{X}$ est projective et lisse. Soit $\overline{W}$ le tore image r\'eciproque de $\overline{V}$ dans $\bG_m^{2g}\subseteq \bA^{2g}$. Alors $\overline{W}$ est param\'etr\'e par
\begin{equation*}
(t_1,t_2,\dots, t_{g+1})\mapsto (t_1,t_2,t_3,\frac{t_1t_2}{t_3},t_4,\frac{t_1t_2}{t_4},\dots, t_{g+1},\frac{t_1t_2}{t_{g+1}}),
\end{equation*}
(voir \cite{CoxLitSch11} \S~1.1) l'application \'etant \'evidemment injective. Notons $(\se_i)_{1\leq i\leq {2g}}$ la base standard de $\bZ^{2g}$ et $M_{\overline{W}}$ (resp. $N_{\overline{W}}$) le sous-module de $\bZ^{2g}$ engendr\'e par $(\se_{2i-1}-\se_{2i})_{1\leq i\leq g}$ et par $\sum_{i=1}^g\se_{2i-1}$ (resp. par $(\se_1+\se_2-\se_{2i-1}-\se_{2i})_{2\leq i\leq g}$). Alors $M_{\overline{W}}$ et $N_{\overline{W}}$ sont en dualit\'e \emph{via} le produit scalaire usuel $(e_i,e_j)=\delta_{ij}$ sur $\bZ^{2g}\times \bZ^{2g}$.

Or, le sous-module $M_{\overline{W}}$ est \'egalement engendr\'e par les lignes de la matrice \`a $g+1$ lignes et $2g$ colonnes
\begin{equation*}
\begin{pmatrix}
1 & 0 & 0 &  1 &  0 &  1 & \cdots & 0 &  1\\
0 & 1 & 0 &  1 &  0 &  1 & \cdots & 0 &  1\\
0 & 0 & 1 & -1 &  0 &  0 & \cdots & 0 &  0\\
0 & 0 & 0 &  0 &  1 & -1 & \cdots & 0 &  0\\
\vdots & \vdots & \vdots & \vdots & \vdots & \vdots & \ddots & \vdots & \vdots \\
0 & 0 & 0 &  0 &  0 &  0 & \cdots & 1 & -1 
\end{pmatrix}
\end{equation*}
correspondant au groupe de caract\'eres $\hat{\overline{W}}$ de $\overline{W}$ (voir \cite{CoxLitSch11} Example 1.1.18 pour le cas $g=2$). 

On en tire que $\hat{\overline{W}}$ s'identifie canoniquement avec $M_{\overline{W}}$ et donc avec $\bZ^{2g}/N_{\overline{W}}$. L'isomorphisme de $\bZ$-modules de $\bZ[\Xi]$ sur $\bZ^{2g}$ qui envoie $\phi_i$ sur $\se_{2i-1}$ et $\overline{\phi}_i$ sur $\se_{2i}$ identifie $I_\Xi\cap J_\Xi$ avec $N_{\overline{W}}$ et donc $\hat{\overline{W}}$ avec $\bZ[\Xi]/(I_\Xi\cap J_\Xi)$ conform\'ement au Lemme~\ref{lem:isohats}. On en tire que $\hat{\overline{V}}$ s'identifie avec $I_\Xi/(I_\Xi\cap J_\Xi)$.

Les r\'esultats de \cite{CoxLitSch11} Chapter 3 permet alors d'\'etablir la bijection entre  $\Sigma(1)$ et $\cP$: la demi-droite $\bR_+e_u$ correspond au sous-espace projectif
\begin{equation*}
X_{2-\frac{u_1+1}{2}}=X_{4-\frac{u_2+1}{2}}=\cdots  = X_{2g-\frac{u_g+1}{2}}\qquad (\text{o\`u $u_i=u_{\phi_i}$, $i\in \{1,2,\dots, g\}$}).
\end{equation*} 
Enfin, la description du diviseur canonique $-\overline{\cD}$ de $\overline{X}$ et de $\cO_{\overline{X}}(\overline{\cD})$ d\'ecoulent des r\'esultats cit\'es dans \cite{CoxLitSch11} (notamment les Chapters 6 et 8).
\end{proof}

Notons $\pi:K\to L^{2g}$ l'application d\'efinie par 
\begin{equation*}
\pi(\alpha)=(\phi_1(\alpha),\overline{\phi}_1(\alpha), \phi_2(\alpha), \overline{\phi}_2(\alpha),Ê\dots, \phi_{g}(\alpha), \overline{\phi}_{g}(\alpha)).
\end{equation*}

\begin{lemm} \label{lemm:adhZar} L'application $\pi$ induit une inclusion de $K^{w,\times}/\bQ^\times$ dans $\overline{V}(L)$ et l'image est Zariski-dense dans $\overline{V}$.  
\end{lemm}

\begin{proof}
Il est clair que $\pi$ induit une inclusion de $K^{w,\times}/\bQ^\times$ dans $\bP^{2g-1}(L)$. D'apr\`es le lemme~\ref{lemm:cnsKw}, son image est contenue dans $\overline{V}(L)$. On en tire que l'adh\'erence de Zariski de l'image est un sous-groupe alg\'ebrique de $\overline{V}$; elle y est \'egale d'apr\`es la correspondance entre tores et $\bZ$-modules libres de type fini.
\end{proof}

\subsection{La compactification $X^K$ de $V^K$} L'action de $\Gamma$ sur $N$ se prolonge en une action sur $N_\bR$. On dit que $\Sigma$ est \emph{$\Gamma$-invariant} si $\gamma(\sigma)\in \Sigma$ pour tout $\gamma\in \Gamma$ et pour tout $\sigma\in \Sigma$. On v\'erifie que l'\'eventail $\Sigma$ est bien $\Gamma$-invariant.  

Dans ce contexte, Voskresenski\u{\i} \cite{Vosk83} a montr\'e que $\overline{X}$ poss\`ede un mod\`ele propre (not\'e $X^K$ ou $X$) sur $\bQ$ et qu'il existe un isomorphisme $X\times_{\Spec{\bQ}}\Spec{L}\simeq \overline{X}$ compatible avec les actions de $\Gamma$. 

Le r\'esultat suivant fournit une description explicite de $X$ comme vari\'et\'e projective, intersection compl\`ete de quadriques dans $\bP^{2g-1}_\bQ$, ainsi que les tores $W$ et $V$.  

\begin{theo}\label{theo:modelesDVX}
Fixons une $\bQ$-base $(\omega_{i})_{1\leq i\leq 2g}$ de $K$ telle que $\omega_1=1$ et $(\omega_i)_{1\leq i\leq g}$ est une $\bQ$-base de $K_0$. Soient $(Q_i)_{1\leq i\leq g}\in \bQ[X_1,X_2,\dots, X_{2g}]$ les polyn\^omes homog\`enes de degr\'e $2$ tels que, pour tout $\alpha=\sum_{i=1}^{2g}{x_i\omega_i}\in K$ (o\`u $(x_1,\dots, x_{2g})\in \bQ^{2g}$), on ait $\alpha\overline{\alpha}=\sum_{i=1}^gQ_i(x_1,x_2,\dots, x_{2g})\omega_i$. 

\case{i} $W$ est $\bQ$-isomorphe au sous-ensemble alg\'ebrique de $\bA^{2g}- \{0\}$ d\'efini par les conditions $Q_1\neq 0$ et $Q_2=\cdots = Q_g=0$.

\case{ii} $V$ est $\bQ$-isomorphe au sous-ensemble alg\'ebrique de $\bP^{2g-1}$ d\'efini par les conditions $Q_1\neq 0$ et $Q_2=\cdots = Q_g=0$.

\case{iii} $X$ est $\bQ$-isomorphe \`a la sous-vari\'et\'e alg\'ebrique de $\bP^{2g-1}$ d\'efini par $Q_2=\cdots = Q_g=0$. Il s'agit d'une intersection compl\`ete. $X$ est la r\'eunion disjointe de $V$ et du sous-ensemble ferm\'e $\cD$ d\'efini par $Q_1=Q_2=\cdots = Q_g=0$.  En outre, $-\cD$ est le diviseur canonique de $X$.

\case{iv}  Il existe un automorphisme $\iota$ de $\bP_L^{2g-1}$ transformant $X\times_{\Spec{\bQ}}\Spec{L}$ sur $\overline{X}$ et identifiant $V\times_{\Spec{\bQ}}\Spec{L}$ avec $\overline{V}$ et $\cD\times_{\Spec{\bQ}}\Spec{L}$ avec $\overline{\cD}$. En outre, l'ensemble $\cP:=\iota^{-1}(\overline{\cP})$ des sous-espaces projectifs de dimension $g-1$ contenu dans $\cD\times_{\Spec{\bQ}}\Spec{L}$ est $\Gamma$-isomorphe \`a l'ensemble $\Sigma(1)$.  

\case{v} Si $F$ est un sous-corps de $\bR$, alors $X(F)=V(F)$.
\end{theo}

\begin{proof} 
\case{i}  Puisque la norme $\alpha\mapsto \alpha\overline{\alpha}$ envoie $K^\times$ dans $K^\times_0$, l'existence et l'unicit\'e des polyn\^omes $Q_i$ est claire. Le fait qu'ils soient homog\`enes de degr\'e $2$ d\'ecoule de la $\bQ$-lin\'earit\'e de la conjugaison complexe et de la $\bQ$-bilin\'earit\'e de la multiplication $(\alpha,\beta)\mapsto \alpha\beta$. Puisque $\omega_1=1$, la condition $\alpha\overline{\alpha}\in \bQ^\times$ \'equivaut, si on pose $\alpha=\sum_{i=1}^{2g}x_i\omega_i$ avec $(x_1,\dots, x_{2g})\in \bQ^{2g}$, l'annulation de $Q_2$, \dots, $Q_g$ en $(x_1,\dots, x_{2g})$. 

\case{ii} C'est une cons\'equence du point \case{i} et de la d\'efinition de $V$.

\case{iii} est une cons\'equence du point \case{ii}, car $X$ est l'adh\'erence de Zariski de $V$ dans $\bP^{2g-1}$.

\case{iv}  On d\'efinit $\iota$ par: si $(x_1,x_2,\dots, x_{2g})\in \bQ^{2g}$ et si $\alpha=\sum_{i=1}^{2g}x_i\omega_i$, alors $\iota(x_1,x_2,\dots, x_{2g})=\pi(\alpha)$. 
Il est clair que $\iota$ d\'efinit alors un application $\bQ$-lin\'eaire injective de $\bQ^{2g}$ dans $L^{2g}$ qui s'\'etend en un isomorphisme de $L$-espaces vectoriels $\bQ^{2g}\otimes_{\bQ}L\simeq L^{2g}$. En composant avec l'identification canonique $L^{2g}=\bQ^{2g}\otimes_{\bQ}L$, on obtient un automorphisme de $L^{2g}$ qui induit donc un automorphisme de $\bP_L^{2g-1}$. Les identifications de $V$ et de $D$ d\'ecoulent de la proposition~\ref{prop:WReouCM} et du lemme~\ref{lemm:adhZar} et la derni\`ere phrase est une cons\'equence des r\'esultats de \cite{Vosk83}.    

\case{v} Il suffit de voir que $\cD(\bR)=\emptyset$. Notons $\tr_{K_0/\bQ}$ la trace de $K_0$ sur $\bQ$. Rappelons que fonction $\alpha\mapsto \tr_{K_0/\bQ}(\alpha\overline{\alpha})$ s'\'etend en une forme quadratique d\'efinie positive sur $\bR\otimes_{\bQ} K$. Mais, avec la notation pr\'ec\'edente:
\begin{equation*}
\tr_{K_0/\bQ}(\alpha\overline{\alpha})=gQ_1(x_1,\dots x_{2g})+\sum_{i=2}^gQ_i(x_1,\dots, x_{2g})\tr_{K_0/\bQ}(\omega_i)
\end{equation*}
Par cons\'equent, le seul \'el\'ement $(x_1,\dots, x_{2g})\in \bR^{2g}$ qui peut se projeter sur un point de $\cD(\bR)$ est $(0,\dots, 0)$ ce qui est impossible.
\end{proof}

\subsection{Un mod\`ele entier} \label{rem:modeleentier}
Soit $F$ un corps de nombres de degr\'e $n$. Rappelons que $\cO_F$ est un $\bZ$-module libre de rang $n$ et que si $E$ est un sous-corps de $F$, il existe un $\bZ$-module $M$ tel que $\cO_F=\cO_E\oplus M$. On peut donc supposer que la famille $(\omega)_{1\leq i\leq 2g}$ figurant dans le th\'eor\`eme~\ref{theo:modelesDVX} soit une $\bZ$-base de $\cO_K$. Il est clair alors que les polyn\^omes $Q_i$ appartiennent \`a $\bZ[X_1,\dots, X_{2g}]$ et que la loi de groupe de $V$ et l'action de $V$ sur $X$ sont d\'efinies sur $\bZ$. On obtient ainsi un mod\`ele projectif de la vari\'et\'e torique $X\subseteq \bP_{\bZ}^{2g-1}$ sur $\Spec{\bZ}$, unique \`a  l'action de $\PGL_{2g}(\bZ)$ pr\`es. Il devient isomorphe sur $\cO_L$ \`a $\overline{X}_\bZ$, le mod\`ele projectif sur $\bZ$ de $\overline{X}$ d\'efini par les m\^emes \'equations $X_1X_2=\dots = X_{2g-1}X_{2g}$. Nous utiliserons syst\'ematiquement ce mod\`ele et cet isomorphisme par la suite. 

\subsection{Hauteurs et fonction zeta} Notons $K^{w,\times}$ le sous-groupe de $K^\times$ form\'e des \'el\'ements $\alpha$ tels que $\alpha\overline{\alpha}\in \bQ^\times$. D'apr\`es le point \case{iv} du lemme~\ref{lem:WVprop}, on dispose d'un isomorphisme naturel $V(\bQ)\simeq K^{w,\times}/\bQ^\times$. 

\begin{lemm}
Tout point $P\in V(\bQ)$ poss\`ede un repr\'esentant $\alpha=\alpha(P)$, unique au signe pr\`es, dans $\cO_K\cap K^{w,\times}$ et qui n'est divisible par aucun entier naturel $>1$. 
\end{lemm}

\begin{proof}
Cela d\'ecoule du fait que $\cO_K$ est un anneau de Dedekind.
\end{proof}

Nous appelons $\alpha=\alpha(P)$ un \emph{repr\'esentant r\'eduit} de $P$ et nous notons $H_{\cD}(P)$ la hauteur (globale) de $P$ associ\'ee au diviseur anticanonique $\cD$ de $X$ (voir \cite{BatTsch95}, Definition 2.1.7 ou \cite{BatTsch98}, Definition 3.5 et aussi \cite{Bourqui11}, \S~3.2). Nous identifions $\cD$ avec le diviseur anticanonique de $\overline{X}$ par le biais de l'isomorphisme de la sous-section~\ref{rem:modeleentier}.  Rappelons que $H_{\cD}(P)$ est d\'efini comme un produit de hauteurs locales 
\begin{equation*}
\prod_{p\leq \infty}H_{\cD,p}(P),
\end{equation*}
o\`u $p$ parcourt l'ensemble des nombres premiers ainsi que  $\infty$. Puisque $\alpha$ est unique au signe pr\`es, le r\'esultat suivant ne d\'epend pas du choix de $\alpha$.

Notons $V_p$ l'ensemble des valuations de $L$ divisant le nombre premier $p$. Rappelons que $\Gamma$ op\`ere transitivement sur $V_p$; si $v\in V_p$, on note $\overline{v}$ l'image de $v$ par conjugaison complexe. Le fait que $c$ soit un \'el\'ement central de $\Gamma$ montre que soit $\overline{v}=v$ pour tout $v\in V_p$ soit $\overline{v}\neq v$ pour tout $v\in V_p$. Pour tout nombre premier $p$, on note $e_p$ le degr\'e de ramification de $p$ dans $L$ et on pose $\nu_p(P)=\min{\{\ord_v(\alpha)\mid v\in V_p\}}$, o\`u $\ord_v(\alpha)$ d\'esigne l'exposant de $v$ dans $\alpha$. 

\begin{prop}\label{prop:hautVK}
Soit $P\in V(\bQ)$ et soit $\alpha$ un repr\'esentant r\'eduit de $P$. Posons $n=\alpha\overline{\alpha}$.  Alors $H_{\cD, \infty}(P)=1$ et, si $p$ est un nombre premier:
\begin{align*}
H_{\cD,p}(P)&=1 \text{ si $p$ est ramifi\'e dans $L$ et si $\overline{v}=v$ pour tout $v\in V_p$},\\
H_{\cD,p}(P)&=p^{\ord_p{n}-2\nu_p(P)/e_p} \text{pour tout autre nombre premier $p$}. 
\end{align*} 
\end{prop}

\begin{rema}
La derni\`ere phrase de la proposition~\ref{prop:compaqVK} d\'ecoule de la proposition~\ref{prop:hautVK}, car $\nu_p(P)=0$ lorsque $p$ est non-ramifi\'e et $n=\prod_p{p^{\ord_p{n}}}$. 
\end{rema}

\begin{proof}
D'apr\`es la formule du produit, $H_{\cD}(P)$ ne d\'epend pas du choix de repr\'esentant $\alpha$ de $P$ dans $K^{w,\times}$. Il suffit donc de calculer les hauteurs locales en supposant $\alpha$ r\'eduit. Si $1\leq i\leq 2g$, \'ecrivons $\psi_i=\phi_{(i+1)/2}$ ou $\overline{\phi}_{i/2}$ selon que $i$ soit impair ou pair. En utilisant le point \case{iv} du th\'eor\`eme~\ref{theo:modelesDVX},  la d\'efinition de la hauteur locale et le fait que $\cO_{\overline{X}}(\overline{D})$ est engendr\'e par $(X_jX_k)/(X_1X_2)$, $1\leq j$, $k\leq 2d$, on trouve
\begin{align*}
H_{\cD,p}(P)^{[L:\bQ]}&=\prod_{v|p}\max_{1\leq j\leq k\leq 2g}\left(\left|\frac{\psi_j(\alpha)\psi_k(\alpha)}{\psi_1(\alpha)\psi_2(\alpha)}\right|_v\right)^{[L_v:\bQ_p]}\\
&=\prod_{v|p}\max_{1\leq j\leq k\leq 2g}\left(\left|\frac{\psi_j(\alpha)\psi_k(\alpha)}{n}\right|_v\right)^{[L_v:\bQ_p]}
\end{align*}
o\`u le produit parcourt l'ensemble des places $v$ de $L$ au dessus de $p$, $|\beta|_v=p^{-\ord_v(\beta)/e_p}$ lorsque $\beta\in L$  et $[L_v:\bQ_p]$ est le degr\'e sur $\bQ_p$ du compl\'et\'e $L_v$ de $L$ en $v$.

\underline{Le cas} $p$ \underline{$=\infty$}. Puisque $n\in \bN^*$, on tire du point \case{i} de la proposition~\ref{prop:WReouCM} que $\left|\frac{\psi_j(\alpha)\psi_k(\alpha)}{n}\right|_v=1$ pour tout $j$, $k$, d'o\`u $H_{\cD,\infty}(P)=1$.  

\underline{Le cas} $p$ \underline{fini}. Il est clair que si $p$ ne divise pas $n$, alors $H_{\cD,p}(P)=1$. Supposons donc que $p$ divise $n$. Deux cas se pr\'esentent.

(1) Tous les $v\in V_p$ sont fix\'es par $c$. Alors la condition $\alpha\overline{\alpha}\in \bN^*$ entra\^{\i}ne que $\ord_v(\alpha)$ est ind\'ependent de $v\in V_p$. On conclut que $|\frac{\psi_j(\alpha)\psi_k(\alpha)}{\psi_1(\alpha)\psi_2(\alpha)}|_v=1$ pour tout $v\in V_p$ et pour tout $j$, $k$, d'o\`u $H_{\cD,p}(P)=1$. Le fait que $\alpha$ soit r\'eduit implique que ce cas ne peut se pr\'esenter que lorsque $p$ est ramifi\'e dans $L$.

(2) Pour tout $v\in V_p$, on a $\overline{v}\neq v$. Alors $\nu_p(\alpha)<e$, car dans le cas contraire, $p$ diviserait $\alpha$, ce que contredit l'hypoth\`ese que $\alpha$ soit r\'eduit. En particulier, si $p$ est non ramifi\'e dans $L$, il existe $v\in V_p$ ne divisant pas $\alpha$. Alors $|\psi_j(\alpha)\psi_k(\alpha)|_v\leq 1$ pour tout $j$, $k$ et, par transitivit\'e de l'action de $\Gamma$ sur $V_p$ et sur $\Hom(L,\bC)$, quel que soit $v\in V_p$, on peut choisir $j=k$ tel que $|\psi_j(\alpha)|_v=1$, d'o\`u $|\psi_j(\alpha)\psi_k(\alpha)|_v=1$ puis $|\frac{\psi_j(\alpha)\psi_k(\alpha)}{n}|_v=|\frac{1}{n}|_v$. On en conclut que $H_{\cD,p}(P)=p^{\ord_p(n)}$. 

Si $p$ est ramifi\'e dans $L$, on voit que $\max_{1\leq j\leq k\leq 2g}\left(\left|\frac{\psi_j(\alpha)\psi_k(\alpha)}{n}\right|_v\right)$ est atteint lorsque $(j,k)$ est choisi de telle mani\`ere que $j=k$ et $\ord_v(\psi_j(\alpha))=\nu_p(P)/e_p$, d'o\`u la formule $H_{\cD,p}(P)=p^{\ord_p(n)-2\nu_p(P)/e_p}$.\end{proof}

\subsection{La repr\'esentation de $\Gamma$ sur $\Pic(X^K)_{\bQ}$} \label{subsec:repPic} Notons $\Pic(X^K)$ le groupe de Picard de $X^K$ et $\Pic(X^K)_\bQ$ le sous-groupe de $\Pic(X^K)$ consistant des \'el\'ements fix\'es par $\Gal(\overline{\bQ}/\bQ)$. 

Notons d'autre part $\Div_{V^K}{X^K}$ le groupe de diviseurs $V^K$-invariants de $X^K$. On sait alors que $\Div_{V^K}{X^K}$ est un groupe ab\'elien libre, engendr\'e par les \'el\'ements de l'ensemble $\cP$ d\'ecrit dans le point \case{iv} du th\'eor\`eme~\ref{theo:modelesDVX} (voir \cite{CoxLitSch11}, \S~4.1). 

\begin{prop}
\case{i} Le $\bQ$-espace vectoriel $\bQ\otimes_\bZ \Pic(X^K)$ est de dimension $2^g-g$.

\case{ii} On dispose d'un $\bQ\Gamma$-isomorphisme naturel entre $\bQ\otimes_\bZ\Pic(X^K)_\bQ$ et le $\bQ$-espace vectoriel $\bQ[\cP]$ de base l'ensemble $\cP$.   
\end{prop}

\begin{proof}
\case{i} Puisque $X^K$ est lisse, on dispose d'une suite exacte de groupes ab\'eliens
\begin{equation*}
0\to M\to \Div_{V^K}{X^K}\to \Pic{X^K}\to 0
\end{equation*}
(voir \cite{CoxLitSch11}, \S\S~4.1 et 4.2). Le r\'esultat est alors clair car $M$ est de rang $g$.

\case{ii} En prenant des $\Gamma$-invariants, on obtient une suite exacte
\begin{equation*}
0\to M^\Gamma\to (\Div_{V^K}{X^K})^\Gamma\to (\Pic{X^K})^\Gamma\to H^1(\Gamma,M).
\end{equation*}
D'apr\`es le lemme~\ref{lem:conjcomp}, $M^\Gamma=\{0\}$. D'autre part, $H^1(\Gamma,M)$ est un groupe fini. Enfin, comme tous les \'el\'ements de $\cP$ sont rationnels sur $L$, on a $(\Pic{X^K})^\Gamma=\Pic{(X^K)}_\bQ$.

On en tire un isomorphisme $\bQ\otimes (\Div_{V^K}{X^K})^\Gamma\simeq \bQ\otimes \Pic{(X^K)}_\bQ$. On conclut en appliquant le th\'eor\`eme~\ref{theo:modelesDVX}.   
\end{proof}

\begin{coro}  \label{coro:dimPicQ}
La dimension de $\bQ\otimes \Pic{(X^K)_\bQ}$ est \'egale au nombre d'orbites de l'action de $\Gamma$ sur $\cP$.
\end{coro}

Soit $\ve\in K$ un \'el\'ement totalement imaginaire (c'est-\`a-dire tel que $c(\ve)=-\ve$) qui engendre $K$ en tant qu'extension de $\bQ$. Un tel \'el\'ement existe selon le lemme~\ref{lem:primKQ}. Pour tout $i\in \{1,2,\dots, g\}$, on pose $\ve_i=\phi_i(\ve)$. Si $\Psi\in \cC\cM_K$, on suppose par la suite que les plongements $\psi_i\in \Psi$ soient num\'erot\'es de telle mani\`ere que $\psi_i(\ve)=u_i\ve_i$ pour tout $i\in \{1,2,\dots, g\}$, o\`u $u_i\in \{-1,1\}$. Rappelons pour m\'emoire la fin de l'\'enonc\'e de la proposition~\ref{prop:XKVKcle}.

\begin{prop} \label{prop:Frob}  
Les $\Gamma$-ensembles $\cP$ et $\cC\cM_K$ sont isomorphes.
\end{prop}

\begin{proof} Tout \'el\'ement de $K$ s'\'ecrit d'une mani\`ere unique sous la forme $x+y\ve$ avec $(x,y)\in K_0^2$. On en tire aussit\^ot que $\cP$ est la r\'eunion des sous-espaces projectifs
\begin{equation*}
X_1+u_1\ve_1X_2=X_3+u_2\ve_2Y_4=\cdots =X_{2g-1}+u_g\ve_gX_{2g}=0, \quad u_i \in \{-1,1\}.
\end{equation*}
On d\'efinit une bijection de $\cC\cM_K$ sur $\cP$ en envoyant $\Psi$ sur l'espace
\begin{equation*}
X_1+\psi_1(\ve)X_2=X_3+\psi_2(\ve)Y_4=\cdots =X_{2g-1}+\psi_g(\ve)X_{2g}=0. 
\end{equation*}
On v\'erifie imm\'ediatement qu'il s'agit d'un $\Gamma$-isomorphisme.
\end{proof}

La description de $\cP$ utilis\'ee dans la d\'emonstration de la proposition~\ref{prop:Frob} fournit probablement la mani\`ere la plus simple de calculer les valeurs de $\rho_K$ figurant dans 

\begin{rema} \label{rem:rhoKcalcul} \emph{D\'emonstration des valeurs de $\rho_K$ apparaissant dans le tab\-leau de la remarque~\ref{rema:rhoKexemples}}\/.  

La plupart de ces valeurs peuvent \^etre extraites de \cite{Dodson84}, mais pour la commodit\'e du lecteur nous en donnerons de d\'emonstrations directes. Les cas o\`u $g\leq 2$ sont faciles. Consid\'erons donc les trois derni\`eres colonnes du tableau. Soit $K_1$ la cl\^oture galoisienne de $K_0$ dans $L$. Puisque $\ve_i^2\in K_0$ pour tout $i$, l'action d'un \'el\'ement de $\Gamma^{K_1}$ sur $\ve_i$ soit le laisse inchang\'e soit le multiplie par $-1$. Il est alors clair que si $\Gamma=2^g.\Gamma_1$, alors $\Gamma^{K_1}$ --- et donc aussi $\Gamma$ --- op\`ere transtivement sur $\cP$, d'o\`u la derni\`ere colonne du tableau. 

Supposons que $\Gamma=C_{2\ell}$.  Alors $K$ est galoisien sur $\bQ$ et on peut l'identifier avec son image dans $\bC$ \emph{via} $\phi_1$. Si $\sigma$ d\'esigne un g\'en\'erateur de $\Gamma$, alors le sous-espace
\begin{equation*}
X_1+\ve X_2=X_3+\sigma^{2}(\ve)X_4=\cdots = X_{2g-1}+\sigma^{2\ell-2}(\ve)X_{2g}
\end{equation*}
est stable par l'action du sous-groupe d'ordre $\ell$ de $\Gamma$ mais non par $c=\sigma^{\ell}$. Son orbite contient donc deux \'el\'ements. Les autres \'el\'ements de $\cP$ ne sont fix\'es ni par $\sigma^2$ ni par $c$: puisque $\ell$ est un nombre premier impair, cela implique toutes les autres orbites sont de cardinal $2\ell$. Cela donne $1+\frac{2^{\ell}-2}{2\ell}$ orbites au total, d'o\`u la colonne du tableau $g=\ell$ avec $\Gamma=C_{2\ell}$. 

Supposons que $\Gamma=D_{2\ell}$. Puisque $\ell$ est impair et $\ve^2\in K_0$, $\prod_{i=1}^\ell\ve_i^2$ est un nombre rationnel n\'egatif et sa racine carr\'ee engendre un corps quadratique imaginaire $M\subseteq L$. Alors $\Gamma=<c>\times \Gamma_0$ o\`u $\Gamma_0\simeq D_\ell$ s'identifie avec $\Gamma^{M}$. Les \'el\'ements d'ordre deux de $\Gamma_0$ op\'erent sur l'ensemble $\{\ve_i^2\mid 1\leq i\leq \ell\}$ en fixant l'un des $\ve_i^2$ et en partitionnant les autres en $\frac{\ell-1}{2}$ paires qu'ils transposent entre elles. Or, pour tout $i$, $L=M(\ve_i^2)$. Il suit que les \'el\'ements d'ordre deux de $\Gamma^M$ op\`ere sur l'ensemble $\{\ve_i\mid 1\leq i\leq \ell\}$ de la m\^eme mani\`ere. On en tire que ces \'el\'ements fixent tous les \'el\'ements de $\cP$. On conclut en prenant pour $\sigma$ un g\'en\'erateur du sous-groupe cylique d'ordre $2\ell$ de $\Gamma$ et en raisonnant comme dans le cas o\`u $\Gamma=C_{2\ell}$.\hfill \qed       
\end{rema}

Notons $\rho^K:\Gamma\to \GL_{2^g}(\bQ)$ (ou $\rho$) la repr\'esentation correspondant \`a l'action de $\Gamma$ sur l'espace vectoriel $\bQ[\cP]$ (ou sur $\bQ[\cC\cM_K]$) et $\theta=\theta^K$ son caract\`ere. Malgr\'e sa simplicit\'e, nous n'avons pas trouv\'e l'\'enonc\'e suivant dans la litt\'erature (voir \cite{Oishi10} pour une contribution r\'ecente \`a l'\'etude de la structure galoisienne de $\cC\cM_K$). 

\begin{theo}  \label{theo:Frob}
Soit $p$ un nombre premier non-ramifi\'e dans $L$ et soit $\sF_p$ le Frobenius associ\'e \`a une place de $L$ divisant $p$. Alors
\begin{equation*}
\theta(\sF_p)=\begin{cases}  2^r \text{ si $p$ se d\'ecompose dans $K$ en $2r$ places dont aucune n'est}\\ \text{\hskip48mm fix\'ee par la conjugaison complexe},\\
                             0  \text{ dans le cas contraire}.
                             \end{cases} 
\end{equation*}
\end{theo}

\begin{proof} Soit $\gamma\in \Gamma$ et soit $M(\gamma)$ la matrice de $\rho(\gamma)$ dans la base de $\bQ[\cC\cM_K]$ form\'ee des \'e\'ements de $\cC\cM_K$. Alors $M(\gamma)$ est une matrice de permutation. Par cons\'equent, $\theta(\gamma)$, la trace de $M(\gamma)$, est \'egal au nombre de types CM fix\'es par $\gamma$. Il s'agit donc de calculer ce nombre dans le cas o\`u $\gamma=\sF_p$.

Posons $\ve_{i+g}=-\ve_i$, ($i\in \{1,2,\dots, g\}$). Notons $\Delta_p$ le sous-groupe de $\Gamma$ engendr\'e par $\sF_p$ et rappelons que la classe de conjugaison de $\Delta_p$ ne d\'epend que de $p$. Puisque $c$ est un \'el\'ement central de $\Gamma$, dire qu'aucune des places de $K$ divisant $p$ ne soit fix\'ee par $c$ \'equivaut \`a dire que $c\notin \Delta_p$. 

Supposons donc que $c\notin \Delta_p$ et que $p\cO_K$ se d\'ecompose en $2r$ places dans $K$ dont aucune n'est stable par conjugaison complexe. Alors l'ensemble $\{\ve_k\}_{1\leq k\leq 2g}$ se d\'ecompose en $2r$ orbites sous $\Delta_p$, et $\ve_k$ et $-\ve_k=\ve_{k+g}$ n'appartiennent pas \`a la m\^eme orbite. Quitte \`a changer le signe de certains des $\ve_k$, on peut donc supposer que $\{\ve_k\}_{1\leq k\leq g}$ se d\'ecompose en $r$ orbites $O_1$, $O_2$, \dots, $O_r$. On en tire facilement qu'un type CM $\Psi=\{\psi_1,\dots, \psi_g\}$ est fix\'e par $\Delta_p$ si et seulement si $\psi_i(\ve)=u_i\ve_i$ avec $u_i=u_j$ lorsque $\ve_i$ et $\ve_j$ appartiennent \`a la m\^eme orbite $O_s$. On en conclut que $\theta(\sF_p)=2^r$.

D'autre part, si $c\in \Delta_p$,  alors $c$ est une puissance de $\sF_p$ car $\Delta_p$ est un groupe cyclique. Comme $c$ ne fixe aucun type CM, il en est de m\^eme pour $\sF_p$, d'o\`u $\theta(\sF_p)=0$.
\end{proof}

Notons $\Orb{(\cC\cM_K)}$ l'ensemble des orbites de l'action de $\Gamma$ sur $\cC\cM_K$ et fixons un syst\`eme de repr\'esentants $\Orb^*{(\cC\cM_K)}\subseteq \cC\cM_K$ de $\Orb{(\cC\cM_K)}$. Si $F$ est un corps de nombres, on note $\zeta_F(s)$ sa fonction zeta de Dedekind.  

\begin{prop}\label{prop:Lrhosprodzeta}
Notons $L(\rho, s)$ la fonction $L$ d'Artin associ\'ee \`a la re\-pr\'esentation $\rho=\rho^K$. Alors
\begin{equation} \label{eq:LrsprodzetahatK}
L(\rho,s)=\prod_{\Phi\in \Orb^*{(\cC\cM_K)}}\zeta_{\hat{K}_\Phi}(s)
\end{equation}
\end{prop}

\begin{proof}
On sait que, si $F$ est un corps de nombres, alors $\zeta_F(s)$ est la fonction d'Artin associ\'ee au tore $T^F$, dont le groupe de caract\`eres est isomorphe \`a $\bZ[\Hom(F,\bC)]$. Si donc $\Phi\in \cC\cM_K$, on tire de la proposition~\ref{prop:OrbPhiHomhatKLiso} que la fonction d'Artin associ\'ee \`a l'orbite de $\Phi$ est \'egale \`a $\zeta_{\hat{K}_\Phi}(s)$ et (\ref{eq:LrsprodzetahatK}) en d\'ecoule aussit\^ot de la proposition~\ref{prop:Frob}.
\end{proof}

\section{Produits eul\'eriens}\label{sec:prodeul}

\subsection{La fonction $Z(K,s)$.} Pour tout $n\in \bN^*$, on pose $\cI_n=\{\ga\in I_K^{w,+}\mid \ga\overline{\ga}=n\cO_K\}$. Consid\'erons, comme dans l'introduction, la s\'erie de Dirichlet
\begin{equation*}
Z(K,s)=\sum_{\ga\in I^{w}_K}\N_{K/\bQ}(\ga)^{-s/g},
\end{equation*}
vue pour le moment comme une s\'erie formelle. Elle s'\'ecrit alors $\sum_{n\geq 1}b_K(n)n^{-s}$ avec $b_K(n)=\sharp\, \cI_n$.

\begin{lemm} \label{lemm:bKmult}
La fonction $n\mapsto b_K(n)$ est multiplicative.
\end{lemm}

\begin{proof}
Rappelons qu'il s'agit de montrer que si $n=n_0n_1$ avec $n_0$, $n_1$ deux entiers positifs premiers entre eux, alors $b_K(n)=b_K(n_0)b_K(n_1)$. Or, $\cO_K$ \'etant un anneau de Dedekind, si $\ga\in I^{w}_K$ et si $\ga\overline{\ga}=n\cO_K$, alors il existe unique $\ga_0$, $\ga_1\in I^{w}_K$ tels que $\ga=\ga_0\ga_1$, $\ga_0\overline{\ga}_0=n_0\cO_K$ et $\ga_1\overline{\ga}_1=n_1\cO_K$. Le lemme en d\'ecoule aussit\^ot.
\end{proof}

Il suit que $Z(K,s)$ poss\`ede un produit eul\'erien. Notons $Z_p(K,s)$ le facteur de ce produit correspondant au nombre premier $p$, de sorte que $Z_p(K,s)=\sum_{k\geq 0}b_K(p^k)p^{-ks}$. 

\begin{lemm} \label{lemm:bKpk} Soit $p$ un nombre premier non ramifi\'e dans $L$.

\case{i} On suppose que $p$ se d\'ecompose dans $K$ en $2r$ places dont aucune n'est fix\'ee par la conjugaison complexe. Si $k\geq 0$ est un entier, alors
\begin{equation*}
b_K(p^k)=(k+1)^r.
\end{equation*}

\case{ii} On suppose qu'au moins une des places de $K$ divisant $p$ est fix\'ee par la conjugaison complexe. Si $2r'$ ($r'\geq 0$) des places divisant $p$ ne sont pas fix\'ees par la conjugaison complexe, alors
\begin{equation*}
b_K(p^k)=\begin{cases} (k+1)^{r'} \text{ si $k$ est pair}\\
0 \text{ si $k$ est impair}.
\end{cases}
\end{equation*}
\end{lemm}

\begin{proof} \case{i} Par hypoth\`ese la factorisation en id\'eaux premiers de $p\cO_K$ s'\'ecrit $\gp_1\overline{\gp}_1\gp_2\overline{\gp}_2\cdots \gp_r\overline{\gp}_r$ avec $\overline{\gp}_i\neq \gp_i$ pour tout $i\in \{1,\dots, r\}$ et $\gp_j\neq \gp_i$, $\overline{\gp}_i$ lorsque $i\neq j$. Soit $\ga\in I^{w}_K$ tel que $\ga\overline{\ga}=p^k\cO_K$. Alors la factorisation de $\ga$ en id\'eaux premiers est de la forme $\ga=\gp_1^{a_1}\overline{\gp}_1^{b_1}\gp_2^{a_2}\overline{\gp}_2^{b_2}\cdots \gp_r^{a_r}\overline{\gp}_r^{b_r}$ avec des entiers positifs $a_i$, $b_i$. La condition $\ga\overline{\ga}=p^k\cO_K$ \'equivaut alors \`a $a_i+b_i=k$ pour tout $i\in \{1,\dots, r\}$. Ainsi, pour tout $i$, $a_i$ peut prendre tous les $k+1$ valeurs $0$, $1$, \dots, $k$ avec $b_i=k-a_i$. On conclut que $b_K(p^k)=(k+1)^r$.

Le point \case{ii} se d\'emontre de la m\^eme mani\`ere.
\end{proof}

\begin{lemm} \label{lemm:PnQnpolys} Notons $n\geq 0$ un entier.

\case{i} Il existe un unique polyn\^ome $P_n(x)\in \bZ[x]$ tel que
\begin{equation*}
\sum_{k\geq 0}(k+1)^nx^k=P_n(x)(1-x)^{-n-1}, \qquad |x|<1.
\end{equation*}
On a $P_0(x)=1$, Si $n\geq 1$, $P_n$ est unitaire de degr\'e $n-1$, $P_n(0)=1$, le coefficient de $x$ dans $P_n$ est $2^n-(n+1)$ et tous ses coefficients sont positifs. 

\case{ii} Il existe un unique polyn\^ome $Q_n(x)\in \bZ[x]$ tel que
\begin{equation*}
\sum_{\ell\geq 0}(2\ell+1)^nx^{2\ell}=Q_n(x)(1-x^2)^{-n-1},  \qquad |x|<1.
\end{equation*}
On a $Q_0(x)=1$. Si $n\geq 1$, $Q_n(0)=1$ et $Q_n$ est un polyn\^ome unitaire en $x^2$ de degr\'e $n$ dont tous ses coefficients sont positifs.
\end{lemm}

\begin{proof} \case{i} Le cas $n=0$ est clair. En g\'en\'eral, on raisonne par r\'ecurrence sur $n$. En \'ecrivant
\begin{equation*}
\sum_{k\geq 0}(k+1)^{n+1}x^k=x\left(\sum_{k\geq 0}(k+1)^nx^k\right)'+\sum_{k\geq 0}(k+1)^nx^k
\end{equation*}
(o\`u ${}'$ d\'esigne la d\'erivation par rapport \`a $x$), on trouve la relation de r\'ecurrence
\begin{equation*}
P_{n+1}(x)=x(1-x)P'_n(x)+(1+nx)P_n(x)
\end{equation*}
ce qui permet de d\'emontrer facilement les propri\'et\'es annonc\'ees des $P_n$.

\case{ii} On raisonne de la m\^eme mani\`ere, la relation de r\'ecurrence \'etant d\'esormais $Q_{n+1}(x)=x(1-x^2)Q'_n(x)+(1+(2n+1)x^2)Q_n(x)$.
\end{proof}

\begin{prop}\label{prop:ZpKs} Soit $p$ un nombre premier non ramifi\'e dans $L$. 

\case{i} On suppose que $p$ se d\'ecompose dans $K$ en $2r$ places dont aucune n'est fix\'ee par la conjugaison complexe.  Alors
\begin{equation*}
Z_p(K,s)=P_r(p^{-s})\left(1-p^{-s}\right)^{-r-1}.
\end{equation*}

\case{ii} On suppose qu'au moins une des places de $K$ divisant $p$ est fix\'ee par la conjugaison complexe $c$. Si $2r'$ ($r'\geq 0$) des places divisant $p$ ne sont pas fix\'ees par $c$, alors
\begin{equation*}
Z_p(K,s)=Q_{r'}(p^{-s})\left(1-p^{-2s}\right)^{-r'-1}.
\end{equation*}
\end{prop}

\begin{proof} La proposition d\'ecoule imm\'ediatement des lemmes~\ref{lemm:bKpk} et \ref{lemm:PnQnpolys}.
\end{proof}

\begin{rema}\label{rema:ZpKszeros} On peut montrer que les polyn\^omes $P_n(x)$ (si $n\geq 3$) et $Q_n(x)$ (si $n\geq 2$) poss\`edent des racines complexes de module $>1$. Cela implique que si $g\geq 3$, certains facteurs eul\'eriens $Z_p(K,s)$ poss\`edent des z\'eros dans le demi-plan $\{\Re(s)>0\}$. Toutefois, les entiers $r$, $r'$ de la proposition~\ref{prop:ZpKs} \'etant major\'es par $g$, on voit que sur un compact donn\'e de $\{\Re(s)>0\}$, $Z_p(K,s)$ ne peut s'annuler que pour un nombre fini de nombres premiers $p$.   
\end{rema}

\begin{theo} \label{theo:prolongZsurL}
Notons $\rho:\Gamma\to \GL_{2^g}(\bQ)$ la repr\'esentation introduite dans le \subsecno~\ref{subsec:repPic} et $L(\rho,s)$ sa fonction d'Artin. Alors la fonction
$s\mapsto Z(K,s)L(\rho,s)^{-1}$
est repr\'esent\'ee par un produit eul\'erien absolument convergent sur le demi-plan $\{\Re(s)>\frac{1}{2}\}$ qui ne s'annule pas sur l'intervalle r\'eel $]\frac{1}{2},+\infty[$.  
\end{theo}

\begin{proof}
Soit $p$ un nombre premier non ramifi\'e dans $L$. Par d\'efinition, le facteur eul\'erien $L_p(\rho,s)$ en $p$ de $L(\rho,s)$ est
\begin{equation*}
L_p(\rho,s)=\det{\left(1-\rho(\sF_p)p^{-s}\right)^{-1}}
\end{equation*}
o\`u, comme pr\'ec\'edemment, $\sF_p$ est le Frobenius d'une place de $L$ divisant $p$. Rappelons que, les racines complexes du polyn\^ome $\det{\left(x-\rho(\sF_p)\right)}$ \'etant de module $1$, $L_p(\rho, s)\neq 0$ lorsque $\Re(s)>0$.

En d\'eveloppant le d\'eterminant, on trouve
\begin{equation*}
L_p(\rho,s)=1+\theta(\sF_p)p^{-s}+O(p^{-2s}),
\end{equation*}
o\`u \`a nouveau $\theta$ d\'esigne le caract\`ere de $\rho$. 

Supposons que $p$ se d\'ecompose en $2r$ places dans $K$, dont aucune n'est fix\'ee par la conjugaison complexe. D'apr\`es le th\'eor\`eme~\ref{theo:Frob}, on a $\theta(\sF_p)=2^r$, d'o\`u $L_p(\rho,s)=1+2^rp^{-s}+O(p^{-2s})$. D'autre part, d'apr\`es le point \case{i} de la proposition~\ref{prop:ZpKs} et le fait que le coefficient de $x$ dans $P_r(x)$ est $2^r-(r+1)$, on voit qu'on a \'egalement $Z_p(K,s)=1+2^rp^{-s}+O(p^{-2s})$. Par cons\'equent, $Z_p(K,s)L_p(\rho,s)^{-1}=1+O(p^{-2s})$.

Supposons qu'il existe une place $v$ de $K$ divisant $p$ qui est stable par conjugaison complexe. Alors les places de $L$ divisant $p$ ont toutes le m\^eme degr\'e r\'esiduel, qui ne peut pas \^etre \'egal \`a $1$ car le degr\'e r\'esiduel de $v$ dans $K/K_0$ est \'egal \`a $2$. Par cons\'equent, $L_p(\rho,s)=1+O(p^{-2s})$. D'apr\`es le point \case{ii} de la proposition~\ref{prop:ZpKs}, on a \'egalement $Z_p(K,s)=1+O(p^{-2s})$. On en tire \`a nouveau que $Z_p(K,s)L_p(\rho,s)^{-1}=1+O(p^{-2s})$.

En utilisant le fait que les entiers $r$, $r'$ de la proposition~\ref{prop:ZpKs} sont major\'es par $g$, on voit que les constantes impliqu\'ees par la notation $O$ peuvent \^etre choisies uniquement en fonction de $K$ (voire de $g$). Soit $\cK$ un compact de $\Re(s)>\frac{1}{2}$. Quitte \`a enlever le nombre fini de facteurs poss\'edant un z\'ero dans $K$ (voir la remarque~\ref{rema:ZpKszeros}), on en tire la convergence uniforme du produit $\prod_{p \text{ non ram}}Z_p(K,s)L_p(\rho,s)^{-1}$ sur $\cK$. Or, cela implique par des r\'esultats bien connus concernant les s\'eries de Dirichlet que le produit est une fonction analytique sur ce demi-plan. 

Supposons d\'esormais que $s\in ]0,+\infty[$. En substituant $x=p^{-s}$ dans la s\'erie $\sum_{k\geq 0}(k+1)^nx^k$ on voit que $Z_p(K,s)\geq 1$. D'autre part, $L_p(\rho, s)\neq 0$. Par cons\'equent, $Z_p(K,s)L_p(\rho,s)^{-1}\neq 0$. D'apr\`es les propri\'et\'es usuelles de produits uniform\'ement convergentes et ce qui a d\'ej\`a \'et\'e \'etabli, cela entra\^{\i}ne que $\prod_{p \text{ non ram}}Z_p(K,s)L_p(\rho,s)^{-1}$ ne s'annule pas lorsque $s>\frac{1}{2}$.  

Soit $p$ un nombre premier ramifi\'e dans $L$. En reprenant la d\'emonstration du lemme, on voit qu'il existe une constante $C>0$ telle que $b_K(p^k)\leq C(k+1)^g$ pour tout $k\geq 1$ puis que $Z_p(K,s)\geq 1$ si $s\in ]0,+\infty[$. \`A nouveau, $L_p(\rho,s)\neq 0$, d'o\`u $Z_p(K,s)L_p(K,s)^{-1}\neq 0$. En rapprochant cela avec ce qui pr\'ec\`ede, on conclut que le produit eul\'erien d\'efinissant $Z(K,s)L(\rho,s)^{-1}$ converge absolument lorsque $\Re(s)>\frac{1}{2}$ et ne s'annule pas lorsque $s\in ]\frac{1}{2},+\infty[$, d'o\`u le th\'eor\`eme.   
\end{proof}

\begin{coro}\label{coro:merZKs}
La fonction $s\mapsto Z(K,s)$ se prolonge en une fonction m\'eromorphe sur le demi-plan $\{\Re(s)> \frac{1}{2}\}$. Elle y est holomorphe sauf en $s=1$, o\`u elle poss\`ede un p\^ole d'ordre \'egal \`a $\rho_K$.
\end{coro}

\begin{proof}
Cela d\'ecoule de la proposition~\ref{prop:Lrhosprodzeta}, de la proposition~\ref{prop:Frob} et du th\'eor\`eme~\ref{theo:prolongZsurL}, ainsi que du fait que la fonction z\^eta de Dedekind d'un corps de nombres a un p\^ole d'ordre un en $s=1$ et est holomorphe sur $\bC- \{1\}$. 
\end{proof}

\subsection{Les fonctions $Z(K,\chi,s)$ et $Z_0(K,s)$} \label{subsection:Zchis} Pour tout caract\`ere $\chi$ du groupe $\cC^{w}_K$, on pose
\begin{equation*}
Z(K,\chi,s)=\sum_{\ga\in I^{w}_K}\chi(\ga)\N_{K/\bQ}(\ga)^{-s/g},
\end{equation*}
de sorte que $Z(K,1,s)=Z(K,s)$. On a alors
\begin{equation*}
Z(K,\chi,s)=\sum_{n\geq 1}b_K(n,\chi)n^{-s},
\end{equation*}
o\`u $b_K(n,\chi)=\sum_{\ga\in \cI_n}\chi(\ga)$. 

\begin{lemm} \label{lemm:bKchimult}
Quel que soit $\chi$, la fonction $n\mapsto b_K(n,\chi)$ est multiplicative.
\end{lemm}

\begin{proof} C'est une g\'en\'eralisation facile de celle du lemme~\ref{lemm:bKmult}.
\end{proof}

Il suit que comme pr\'ec\'edemment, $Z(K,\chi,s)$ poss\`ede un produit eul\'erien. Notons $Z_p(K,\chi,s)$ le facteur correspondant au nombre premier $p$, c'est-\`a-dire $Z_p(K,\chi,s)=\sum_{k\geq 0}b_K(p^k,\chi)p^{-ks}$.

Si $\Phi\in \cC\cM_K$, on note $\chi_{\Phi}$ le caract\`ere $\chi\circ \N_{\hat{\Phi}}$ de $\Pic(\cO_{\hat{K}_\Phi})$ (voir le \subsecno~\ref{subsection:typesCM}). Pour des r\'esultats semblables \`a la proposition qui suit, voir aussi \cite{Odoni91}, \S~2 et \cite{BoxGru15}, Proposition~2.3.

\begin{prop} \label{prop:gaPhigP}
Soit $p$ un nombre premier non ramifi\'e dans $K$. 

\case{i} Si $\Phi\in \cC\cM_K$ et si $\gP$ est un id\'eal premier de degr\'e un de $\hat{K}_\Phi$ divisant $p$, alors $\N_{\hat{\Phi}}\gP\in \cI_p$.

\case{ii} R\'eciproquement, si $\ga\in \cI_p$, alors il existe $\Phi\in \cC\cM_K$ et un id\'eal premier $\gP$ de degr\'e un de $\hat{K}_\Phi$ divisant $p$ tel que $\N_{\hat{\Phi}}\gP=\ga$. 

\case{iii} Soit $\Phi$ (resp. $\Phi'$) $\in \cC\cM_K$ ayant la propri\'et\'e qu'il existe un id\'eal premier $\gP$ (resp. $\gP'$)  de degr\'e un de $\hat{K}_\Phi$ (resp. de $\hat{K}_{\Phi'}$) divisant $p$ tel que $\N_{\hat{\Phi}}\gP=\N_{\hat{\Phi'}}\gP'=\ga$. Alors il existe $\tau\in \Gamma$ tel que $\tau\cdot \Phi=\Phi'$ et $\tau(\gP)=\gP'$. 
\end{prop}

\begin{proof}  Rappelons que $L$ d\'esigne la cl\^oture galoisienne de $K$ dans $\bC$.

\case{i}  Posons $\ga=\N_{\hat{\Phi}}(\gP)$. Alors 
\begin{align*}(\ga\cO_L)(\overline{\ga}\cO_L)&=\prod_{\psi\in \hat{\Phi}}\psi(\gP)\cO_L\prod_{\psi\in \hat{\Phi}}\overline{\psi}(\gP)\cO_L\\
&=\prod_{\psi\in \Hom(\hat{K},\bC)}\psi(\gP)\cO_L=(\N_{\hat{K}/\bQ}\gP)\cO_L=p\cO_L,
\end{align*} 
d'o\`u $\ga\overline{\ga}=p\cO_K$. 

\case{ii} Notons $\Sp(\ga)$ l'ensemble des id\'eaux premiers divisant $\ga$. 

Fixons $\gQ\in \Sp(\ga)$ et rappelons que $\Gamma$ op\`ere transitivement sur l'ensemble des id\'eaux premiers de $L$ divisant $p$. Remarquons \'egalement que les id\'eaux $\ga$ et $\overline{\ga}$ sont \'etrangers entre eux, car si $\gp$ \'etait un id\'eal premier de $K$ divisant \`a la fois $\ga$ et $\overline{\ga}$, alors $\gp^2$ diviserait $\ga\overline{\ga}=p\cO_K$, contrairement \`a l'hypoth\`ese que $p$ ne soit pas ramifi\'e dans $K$. 

Posons donc 
\begin{equation*}
S_\gQ=\{\gamma\in \Gamma\mid \gQ \text{ divise } \gamma(\ga\cO_L)\}
\end{equation*}
et \'ecrivons $S=S_\gQ$ pour abr\'eger. En utilisant ce qui pr\'ec\`ede, on voit que $S$ est un type CM sur $L$. En plus, $\ga$ \'etant un id\'eal de $K$, on a $S\gamma=S$ pour tout $\gamma\in \Gamma^K$. Par cons\'equent, $S$ est le prolongement \`a $L$ d'un type CM $\Phi$ sur $K$. Notons $\hat{K}$ le corps reflex et $\hat{\Phi}$ le type CM reflex sur $\hat{K}$. Posons $\gP=\gQ\cap \cO_{\hat{K}}$. En rappelant que $\Gamma^{\hat{K}}=\{\gamma\in \Gamma\mid \gamma S=S\}$, on v\'erifie que $\N_{\hat{\Phi}}(\gP)=\ga$ et que $\gP$ est de degr\'e un.  

\case{iii} Notons $\gQ$ (resp. $\gQ'$) un \'element de $\Sp(\ga)$ divisant $\gP$ (resp. $\gP'$) et \'ecrivons $S=S_{\gQ}$, $S'=S_{\gQ'}$. Alors $\Sp(\ga)=S^{-1}\cdot \gQ$ et, si $\tilde S$ d\'esigne le prolongement de $\Phi$ \`a $L$, on a encore $\Sp(\ga)=\tilde{S}^{-1}\cdot \gQ$, d'o\`u $S=\tilde{S}$. De m\^eme, $S'$ est le prolongement \`a $L$ de $\Phi'$. Soit $\tau\in \Gamma$ tel que $\tau(\gQ)=\gQ'$. Alors $S'=\tau\cdot S$, d'o\`u $\Phi'=\tau\cdot \Phi$, $\Gamma^{\hat{K}_{\Phi'}}=\{\gamma\in \Gamma\mid \gamma\cdot S'=S'\}=\Gamma^{\tau(\hat{K}_\Phi)}$ puis $\hat{K}_{\Phi'}=\tau(\hat{K}_\Phi)$. Il suit que $\tau(\gQ)$ divise les deux id\'eaux premiers $\gP'$ et $\tau(\gP)$ de $\hat{K}_{\Phi'}$, d'o\`u $\gP'=\tau(\gP)$.  
\end{proof}

Si $\chi:\cC^w_K\to \bC^\times$ est un caract\'ere et si $\Phi\in \cC\cM_K$, on note $\chi_{\hat{\Phi}}$ le caract\'ere $\chi\circ \N_{\hat{\Phi}}:\Pic(\cO_{\hat{K}})\to \bC^\times$. On note $L(\chi_{\hat{\Phi}},s)$ la s\'erie $L$ de Hecke associ\'ee, en identifiant $\chi_{\hat{\Phi}}$ avec un caract\'ere de $\Gal(\overline{\hat{K}}/\hat{K})$ \`a l'aide de la loi de r\'eciprocit\'e d'Artin. Rappelons que $\Orb^*(\cC\cM_K)$ d\'esigne un syst\`eme de repr\'esentants des orbites de l'action de $\Gamma$ sur $\cC\cM_K$. On remarquera que si $\Phi$ et $\Phi'$ sont dans le m\^eme orbite, alors $L(\chi_{\hat{\Phi}},s)=L(\chi_{\hat{\Phi'}},s)$. Posons
\begin{equation*}
\Pi(K,\chi,s)=\prod_{\Phi\in \Orb^*(\cC\cM_K)}{L(\chi_{\hat{\Phi}},s)}.
\end{equation*}
Les consid\'erations qui suivent ne d\'ependent pas du choix des repr\'esentants. 

\begin{coro}\label{coro:quothol}
Soit $\chi:\cC^w_K\to \bC^\times$ un caract\'ere. Alors la fonction
\begin{equation*}
s\mapsto Z(K,\chi, s)\Pi(K,\chi,s)^{-1}
\end{equation*}
s'\'ecrit comme produit eul\'erien absolument convergent sur $\{\Re(s)>\frac{1}{2}\}$.
\end{coro}

\begin{proof} Il suffit de montrer que, si $p$ est non-ramifi\'e dans $K$, alors les coefficients de $p^{-s}$ dans $Z_p(K,\chi,s)$ et dans $\prod_{\Phi\in \Orb^*(\cC\cM_K)}{L_p(\chi_{\hat{\Phi}},s)}$ sont \'egaux. En effet, dans ce cas, on a
\begin{equation*}
Z_p(K,\chi,s)\Pi(K,\chi,s)^{-1}\in 1+O(p^{-2s})
\end{equation*} 
et on conclut alors par un argument semblable \`a celui utilis\'e dans la d\'emonstration du th\'eor\`eme~\ref{theo:prolongZsurL}.

Par d\'efinition, le coefficient de $p^{-s}$ dans $Z_p(K,\chi,s)$ est \'egal \`a $\sum_{\ga\in \cI_p}\chi(\ga)$. D'autre part, d'apr\`es la proposition~\ref{prop:gaPhigP}, il y a une correspondance biunivoque entre l'ensemble $\cI_P$ et l'ensemble des couples $(\Phi,\gP)$ o\`u $\Phi\in \Orb^*(\cC\cM_K)$ et $\gP$ un id\'eal premier $\hat{K}_\Phi$ de degr\'e un; si l'id\'eal $\ga$ correspond au couple $(\Phi,\gP)$, alors $\ga=\N_{\hat{\Phi}}\gP$. Puisque $\chi(\ga)=\chi_{\hat{\Phi}}(\gP)$, on en tire que 
\begin{equation*}
\sum_{\ga\in \cI_p}\chi(\ga)=\sum_{\Phi\in \Orb^*(\cC\cM_K)}\left(\sum_{\gP\in \cJ_p(\hat{K}_{\Phi})}\chi_{\hat{\Phi}}(\gP)\right),  \tag*{$(\dag)$}
\end{equation*}
o\`u la somme int\'erieure parcourt l'ensemble $\cJ_p(\hat{K}_\Phi)$ des id\'eaux premiers de $\hat{K}_\Phi$ divisant $p$ et de degr\'e un. 
Mais $\sum_{\gP\in \cJ_p(\hat{K}_\Phi)}\chi_{\hat{\Phi}}(\gP)$ est \'egal au coefficient de $p^{-s}$ dans $L_p(\chi_{\hat{K}},s)$ et donc le membre droit de $(\dag)$ est \'egal au coefficient de $p^{-s}$ dans $\prod_{\Phi\in \Orb^*(\cC\cM_K)}{L_p(\chi_{\hat{\Phi}},s)}$ d'o\`u le corollaire. 
\end{proof}

\begin{coro}\label{coro:ZKchismero} Notons encore $\chi:\cC^w_K\to \bC^\times$ un caract\`ere. Alors la fonction $s\mapsto Z(K,\phi,s)$ se prolonge en une fonction m\'eromorphe sur $\{\Re(z)>\frac{1}{2}\}$, holomorphe sauf \'eventuellement en $s=1$. En $s=1$, elle a un p\^ole d'ordre \'egal au nombre d'orbites de l'action de $\Gamma$ sur $\cC\cM_K$ pour lesquelles $\chi_{\hat{\Phi}}$ est le caract\`ere trivial.
\end{coro}

\begin{proof}
Une fonction de la forme $L(\chi_{\hat{\Phi}},s)$ \'etant une fonction de Hecke associ\'ee \`a une repr\'esentation de dimension un, elle est holomorphe sur $\{\Re(s)>\frac{1}{2}\}$ sauf \'eventuellement en $s=1$. Elle est holomorphe en $s=1$ si et seulement si $\chi_{\hat{\Phi}}$ est non trivial; dans le cas contraire $L(\chi_{\hat{\Phi}},s)$ est \'egale \`a la fonction z\^eta de Dedekind $\zeta_{\hat{K}}(s)$ de $\hat{K}$, qui a un p\^ole d'ordre un en $s=1$. Le corollaire d\'ecoule alors du corollaire~\ref{coro:quothol}.   
\end{proof}

Notons $t_K:I_K\to I_K^w$ le morphisme de groupes d\'efini par $t_K(\gb)=\gb\overline{\gb}^{-1}$. Il est clair que l'image de $t_K$ est contenue dans le sous-groupe $\{\ga\in I^w_K\mid \ga\overline{\ga}=\cO_K\}$ de $I^w_K$. On note \'egalement $t_K$ le morphisme induit $\Pic(\cO_K)\to \cC^w_K$ et $\cC^0_K$ l'image de $\Pic(\cO_K)$ par $t_K$.

\begin{prop}\label{prop:CwKfini} Le groupe $\cC^w_K$ est fini. Le groupe quotient $\cC^w_K/\cC^0_K$ est isomorphe \`a un produit fini de groupes cycliques d'ordre $2$.
\end{prop}

\begin{proof} Le groupe $\Pic(\cO_K)$ \'etant fini, il suffit de d\'emontrer la seconde affirmation. Pour cela, il suffit de montrer que le groupe $(\cC^w_K)^2=\{\gamma^2\mid \gamma\in \cC^w_K\}$ est contenu dans $\cC^0_K$. Soit donc $\gamma\in \cC^w_K$ et soit $\ga\in I^{w,+}_K$ un repr\'esentant de $\gamma$. Soit $n>0$ l'entier tel que $\ga\overline{\ga}=n\cO_K$. Posons $\gc=\ga^2/n$. Alors $\gc\in \gamma^2$ et $\gc\overline{\gc}=\cO_K$. Si $\gb$ est l'id\'eal num\'erateur de $\gc$, alors $\gc=t_K(\gb)$. Par cons\'equent,  $\gamma^2\in \cC^0_K$, d'o\`u le r\'esultat.
\end{proof}

\begin{prop}\label{prop:Z0prlong} Le fonction $Z_0(K,s)$ (d\'efinie dans l'introduction) se prolonge m\'eromorphiquement sur $\{\Re(s)>\frac{1}{2}\}$. Elle y est holomorphe sauf en $s=1$ o\`u elle a un p\^ole d'ordre $\rho_K$.
\end{prop}

\begin{proof}D'apr\`es le corollaire~\ref{coro:ZKchismero}, chacune des fonctions $Z(K,\chi, s)$ se prolonge en une fonction m\'eromorphe sur $\{\Re(s)>\frac{1}{2}\}$, holomorphe en dehors de $s=1$. La fonction $Z_0(K,s)$ jouit donc des m\^emes propri\'et\'es. Pour conclure, il suffit de calculer l'ordre du p\^ole de $Z_0(K,s)$ en $s=1$. Celui-ci est major\'e par les ordres des p\^oles des fonctions $Z(K,\chi,s)$ et ces ordres atteignent leur maximum aux caract\`eres $\chi$ tels que tous les $\chi_{\hat{\Phi}}$ sont triviaux (par exemple lorsque $\chi$ est lui-m\^eme trivial). Dans ce cas, on sait d'apr\`es la proposition~\ref{prop:Lrhosprodzeta} et le th\'eor\`eme~\ref{theo:prolongZsurL}, que l'ordre est \'egal au rang de $\Pic(X^K)_\bQ$. Pour conclure, il suffit alors de d\'emontrer le lemme qui suit.
\end{proof}

Rappelons que $h_K^w$ d\'esigne l'ordre de $\cC_K^w$.

\begin{lemm}\label{lemm:Z0KminZK} Il existe une constante $C_K>0$ telle que pour tout $\sigma>1$ on ait 
\begin{equation*}
Z_0(K,\sigma)\geq \frac{w_K}{h^w_KC_K^{\sigma}}Z(K,\sigma).
\end{equation*}
\end{lemm}

\begin{proof}
Abr\'egeons $h^w_K$ en $h$. Soit $\gb_i$ ($1\leq i\leq h$) un syst\`eme de repr\'esentants des \'el\'ements de $\cC^w_K$ choisis dans $I^{w,+}_K$. Pour tout $i\in \{1,\dots, h\}$, on note $\beta_i$ un g\'en\'erateur de $\gb_i^h$ tel que $\beta\overline{\beta}\in \bQ^\times$, et on pose $J_i=\{\alpha\in \gb_i^{-1}\mid \alpha\overline{\alpha}\in \bQ^\times\}$. Alors
\begin{align*}
Z(K,\sigma)\!=\! &\sum_{\gamma\in \cC^w_K}\!\sum_{\ga \in I_K^{w,+}\cap \gamma}\frac{1}{\N_{K/\bQ}(\ga)^{\sigma/g}}\!=\!\frac{1}{w_K}\sum_{i=1}^h\frac{1}{\N_{K/\bQ}(\gb_i)^{\sigma/g}}\sum_{\alpha\in J_i}\frac{1}{\N_{K/\bQ}(\alpha)^{\sigma/g}}\\
&=\frac{1}{w_K}\sum_{i=1}^h\N_{K/\bQ}(\gb_i)^{(h-1)\sigma/g}\sum_{\alpha\in J_i}\frac{1}{\N_{K/\bQ}(\alpha\beta_i)^{\sigma/g}}.
\end{align*}
Posons $C_K=\max_i\{\N_{K/\bQ}(\gb_i)^{(h-1)/g}\}$. Puisque $\alpha\beta_i\in \cO_K$ pour tout $i$ et pour tout $\alpha\in J_i$, on voit que
\begin{equation*}
\N_{K/\bQ}(\gb_i)^{(h-1)\sigma/g}\sum_{\alpha\in J_i}\frac{1}{\N_{K/\bQ}(\alpha\beta_i)^{\sigma/g}}\leq C_K^{\sigma}Z_0(K,\sigma)
\end{equation*}
et le lemme en d\'ecoule en prenant la somme sur $i$.
\end{proof}

\subsection{Les fonctions $\tilde{Z}(K,s)$, $\tilde{Z}(K,\chi,s)$ et $Z_h(X^K,s)$} Si $p$ est un nombre premier, on note $e_p$ le degr\'e de ramification de $p$ dans $L$. Un id\'eal fractionnaire $\ga$ de $K$ est dit \emph{r\'eduit} si $\ga\subseteq \cO_K$ et si $\ga$ n'est divisible par aucun entier naturel $>1$. Si $\ga\subseteq \cO_K$ est un id\'eal, on pose 
\begin{equation*}
\nu_p(\ga)=\min{\{\ord_v(\ga)\mid v\in V_p \}}
\end{equation*}
et, si $n\in \bN^*$, on pose $\cR_n=\{\ga\in \cI_n\mid \ga \text{ r\'eduit}\}$. Si $\chi:\cC^w_K\to \bC^\times$ est un caract\`ere, on d\'efinit la fonction $\tilde{Z}(K,\chi,s)$ par le produit eul\'erien $\tilde{Z}(K,\chi, s)=\prod_{p \text{ premier}}\tilde{Z}_p(K,,\chi, s)$, o\`u $\tilde{Z}_p(K,\chi, s)=Z_p(K,\chi, s)$ lorsque $p$ est non ramifi\'e et, lorsque $p$ est ramifi\'e:
\begin{equation*}
\tilde{Z}_p(K,\chi,s)=\begin{cases} \left(1-p^{-2s}\right)^{-1} \text{ si $v=\overline{v}$ pour tout $v\in V_p$}\\
  \left(1-p^{-2s}\right)^{-1}\sum_{k\geq 0}\left(\sum_{\ga\in \cR_{p^k}}\chi(\ga)p^{2\nu_{p}(\ga)s/e_p}\right)p^{-ks}  \text{ sinon.}
  \end{cases}
\end{equation*}
 Notons que $\nu_p(\ga)< e_p$ lorsque $\ga$ est r\'eduit. Puisque $e_p\geq 2$ et le cardinal de $\cR_{p^k}$ est major\'e par une puissance de $k$, la s\'erie apparaissant dans la d\'efinition de $\tilde{Z}_p(K,\chi, s)$ lorsque $p$ est ramifi\'ee converge pour tout $s$ avec $\Re(s)>0$.  On pose $\tilde{Z}(K,s)=\tilde{Z}(K,\chi,s)$. 

\begin{prop}\label{prop:tildeZKs} \case{i} La fonction
\begin{equation*}
s\mapsto \tilde{Z}(K,s)L(\rho, s)^{-1}
\end{equation*}
se prolonge en une fonction analytique sur le demi-plan $\{\Re(s)>\frac{1}{2}\}$ et ne s'annule pas sur l'intervalle r\'eel $]\frac{1}{2},+\infty[$.

\case{ii} La fonction $s\mapsto \tilde{Z}(K,s)$ se prolonge en une fonction m\'eromorphe sur $\{\Re(s)>\frac{1}{2}\}$. Elle est holomorphe sauf en $s=1$ o\`u elle a un p\^ole d'ordre $\rho_K$.
\end{prop}

\begin{proof}
La d\'emonstration du point \case{i} est la m\^eme que celle du th\'eor\`eme~\ref{theo:prolongZsurL}, en remarquant que les facteurs eul\'eriens au premiers ramifi\'es prennent \`a nouveau des valeurs positives sur $]0,+\infty[$. Le point \case{ii} est alors une cons\'equence du point \case{i} (voir le corollaire \ref{coro:merZKs}). 
\end{proof}

\begin{prop}\label{prop:tildeZKchis} \case{i} La fonction
\begin{equation*}
s\mapsto \tilde{Z}(K,\chi, s)\Pi(K,\chi,s)^{-1}
\end{equation*}
se prolonge en une fonction analytique sur le demi-plan $\{\Re(s)>\frac{1}{2}\}$.

\case{ii} La fonction $s\mapsto \tilde{Z}(K,\chi, s)$ se prolonge en une fonction m\'eromorphe sur $\{\Re(s)>\frac{1}{2}\}$. Elle y est holomorphe sauf en $s=1$ o\`u elle a un p\^ole d'ordre au plus \'egal au nombre d'orbites de l'action de $\Gamma$ sur $\cC\cM_K$ pour lesquelles $\chi_{\hat{\Phi}}$ est le caract\`ere trivial.
\end{prop}

\begin{proof}
\case{i} Puisque $\tilde{Z}_p(K,\chi,s)=Z_p(K,\chi,s)$ lorsque $p$ est non ramifi\'e dans $L$ et les facteurs eul\'eriens $\tilde{Z}_p(K,\chi,s)$ convergent pour $\Re(s)>\frac{1}{2}$ m\^eme lorsque $p$ est ramifi\'e, le r\'esultat d\'ecoule du corollaire~\ref{coro:quothol}.

\case{ii} On le d\'eduit du point \case{i} par l'argument d\'ej\`a utilis\'e dans la d\'emonstration du corollaire~\ref{coro:ZKchismero}.
\end{proof}

Posons 
\begin{equation*}
\tilde{Z}_0(K,s)=\frac{w_K}{h_K^w}\sum_{\chi}\tilde{Z}(K,\chi,s),
\end{equation*}
o\`u \`a nouveau la somme parcourt l'ensemble des caract\`eres $\chi:\cC^w_K\to \bC^\times$.
Rappelons que $Z_h(X^K,s)$ d\'esigne la fonction z\^eta des hauteurs associ\'ee \`a la classe anticanonique de $X^K$. 

\begin{prop} \label{prop:ZhManin}
On a 
$Z_h(X^K,s)=\tfrac{1}{2}\tilde{Z}_0(K,s)\zeta(2s)^{-1}$.
En outre, la fonction $Z_h(X^K,s)$ se prolonge en une fonction m\'eromorphe sur $\{Re(s)>\frac{1}{2}\}$. Elle est holomorphe sauf en $s=1$ o\`u elle poss\`ede un p\^ole d'ordre $\rho_K$. 
\end{prop} 

\begin{proof} La premi\`ere formule d\'ecoule de la proposition~\ref{prop:hautVK}, en rappelant que tout point de $V^K$ poss\`ede un repr\'esentant r\'eduit unique au signe pr\`es. Pour le reste, il suffit de montrer que $\tilde{Z}_0(K,s)$ poss\`ede les propri\'et\'es requises. Le fait que $\tilde{Z}_0(K,s)$ soit m\'eromorphe sur $\{\Re(s)>\frac{1}{2}\}$, holomorphe sauf un p\^ole en $s=1$ d'ordre au plus $\rho_K$ d\'ecoule des proposition~\ref{prop:tildeZKs} et \ref{prop:tildeZKchis}. Enfin, on montre par un argument semblable \`a celui utilis\'e pour \'etablir le lemme~\ref{lemm:Z0KminZK} qu'il existe une constante $C'_K>0$ telle que, pour tout $\sigma>1$, on ait
\begin{equation*}
\tilde{Z}_0(K,\sigma)\geq \frac{w_K}{h_K^wD_K^\sigma}\tilde{Z}(K,\sigma), 
\end{equation*}
ce qui suffit pour montrer que le p\^ole de $s=1$ est bien d'ordre $\rho_K$. 
\end{proof}

\subsection{La constante $c_K$ et la relation avec d'autres d\'emonstrations de la conjecture de Manin pour les vari\'et\'es toriques} \label{subsec:liensautres} Notons $\theta_K$ (resp. $\tilde{\theta}_K$) le coefficient de $(s-1)^{-\rho_K}$ dans le d\'eveloppement de Laurent de $Z_0(K,s)$ (resp. de $\tilde{Z}_0(K,s)$). Il est bien connu que $c_K=\frac{\theta_K}{(\rho_K-1)!}$. Soit $\cS_K$ l'ensemble des caract\`eres $\chi:\cC^w_K\to \bC^\times$ tels $Z(K,\chi,s)$ ait un p\^ole d'ordre $\rho_K$ en $s=1$. Alors $\chi\in \cS_K$ si et seulement si  $\chi_{\hat{\Phi}}=1$ pour tout $\Phi\in \Orb^*(\cC\cM_K)$ (ce qui montre que $\cS_K$ est un groupe), ou encore si et seulement si $\prod_{\Phi\in \Orb^*(\cC\cM_K)}L(\chi_{\hat{\Phi}},s)=L(\rho,s)$. On tire alors du corollaire~\ref{coro:quothol} et du th\'eor\`eme~\ref{theo:prolongZsurL} que 
\begin{equation}  \label{eq:thetaK}
\theta_K=\frac{w_KL^*(\rho,1)}{h_K^w}\sum_{\chi\in \cS_K}\left(\frac{Z(K,\chi,s)}{L(\rho,s)}\right)_{s=1},   
\end{equation}
o\`u $L^*(\rho,1)=\lim_{s\to 1}(s-1)^{\rho_K}L(\rho, s)$ est \'egal au produit des r\'esidus en $s=1$ des fonctions zeta $\zeta_{\hat{K}}(s)$ lorsque $\Phi$ parcourt $\Orb^*(\cC\cM_K)$ (voir la proposition~\ref{prop:Lrhosprodzeta}), et les quotients $\frac{Z(K,\chi,s)}{L(\rho,s)}$ sont des produits eul\'eriens absolument convergentes sur $\{\Re(s)>\frac{1}{2}\}$. On obtient une formule semblable pour $\tilde{\theta}_K$, o\`u chacune des fonctions $Z(K,\chi,s)$ est remplac\'ee par $\tilde{Z}(K,\chi,s)$.  Ainsi, au moins lorsque $\cS_K$ est trivial, on obtient des formules pour $c_K$, $\theta_K$ et $\tilde{\theta}_K$ qui ressemblent \`a celles d\'ecrites en termes de nombres de Tamagawa par Peyre \cite{Peyre95}, puis d\'emontr\'ees dans \cite{BatTsch95} et \cite{BatTsch98}. 

Toutefois, $\cS_K$ n'est pas trivial en g\'en\'eral. Notons $I^{w,\myll}_K$ le sous-groupe de $I^w_K$ form\'e des id\'eaux $\ga$ \emph{localement libres}, c'est-\`a-dire tels que, pour tout nombre premier $p$, le $\bZ_p\otimes_{\bZ}\cO_K$-module $\bZ_p\otimes \ga$ poss\`ede un g\'en\'erateur $a_p$  avec $a_p\overline{a}_p\in \bQ_p^\times$; notons \'egalement $\cC^{w,\myll}_K$ l'image de  $I^{w,\myll}_K$ dans $\cC^w_K$. On voit facilement que $\N_{\hat{\Phi}}(\Pic(\cO_{\hat{K}}))\subseteq \cC^{w,\myll}_K$ pour tout $\Phi\in \cC\cM_K$. Mais en g\'en\'eral, $\cC^{w,\myll}_K\neq \cC^w_K$ lorsque $[K:\bQ]\geq 4$. (Il est facile de construire des exemples explicites avec $[K:\bQ]=4$ et $K_0$ un corps quadratique dont tous les \'el\'ements de $ \cO_{K_0}^\times$ sont de norme $1$.) Dans ces circonstances, $\cS_K\neq \{1\}$ et la somme dans le membre droite de \ref{eq:thetaK} contient plusieurs termes. 

Par ailleurs, il n'est pas difficile de montrer que $\cC^{w,\myll}_K$ s'identifie avec le groupe des classes (au sens d'Ono \cite{Ono61}) du tore $V^K$.

La lecture de la d\'emonstration du th\'eor\`eme~3.4.6 de \cite{BatTsch95} pourrait laisser supposer un lien entre le groupe $\cS_K$ et (le dual) du groupe --- not\'e $A(T)$ dans \cite{BatTsch95} avec ici le $\bQ$-tore $T$ \'egal \`a $V^K$ --- mesurant le {\og}d\'efaut d'approximation faible{\fg} de $T$. Par d\'efinition, $A(T)$ est le groupe quotient $T(\bA_\bQ)/\overline{T(\bQ)}$, o\`u $T(\bA_{\bQ})$ est le groupe ad\'elique de $T$ et $\overline{T(\bQ)}$ est l'adh\'erence dans $T(\bA_\bQ)$ pour la topologie produit (voir par exemple \cite{BatTsch95}, page 605), et $T$ satisfait \`a l'approximation faible si et seulement si $A(T)=\{1\}$.\footnote{Il est plus habituel de d\'efinir l'approximation faible comme \'etant la propri\'et\'e que $T(\bQ)$ soit dense dans le produit $\prod_{p\leq \infty}T(\bQ_p)$ pour la topologie produit (voir par exemple \cite{Bourqui11}, \S~2.4.3). Puisque $T(\bA_\bQ)$ est dense dans $\prod_{p\leq \infty}T(\bQ_p)$, les deux d\'efinitions sont \'equivalentes.}  Or, il n'en est rien, car alors que $\cS_K$ peut \^etre non-trivial, nous allons d\'emontrer la proposition suivante:---

\begin{prop} \label{prop:VKAF}
Quel que soit le corps CM $K$, les tores $W^K$ et $V^K$ satisfont \`a l'approximation faible.
\end{prop} 

\begin{proof} Consid\'erons le cas de $W^K$, le cas de $V^K$ en \'etant une cons\'equence facile. Rappelons que si $E$ est une $\bQ$-alg\`ebre, alors $W^K(E)$ s'identifie avec l'ensemble $\{\alpha\in E\otimes_{\bQ}K\mid \N_{K/K_0}\alpha\in E^\times\}$. Lorsque $E=\bQ_p$ ou $\bR$, cette identification est un hom\'eomorphisme pour les topologies usuelles. Pour tout nombre premier $p$ et pour $p=\infty$, on pose
\begin{equation*}
K^{\times,w}_p=\{(\alpha_v)_v\in \prod_{v\in V_p}K_v^\times \mid\N_{K/K_0}\alpha_v \in \bQ_{p}^\times \text{ ind\'ependent de }v\}
\end{equation*}
et on note $\cO^{\times,w}_{K,p}$ le sous-groupe de $K^{\times,w}_p$ form\'e des \'el\'ements $(\alpha_v)$ avec $\alpha_v\in \cO_{K,v}^\times$ pour tout $v\in V_p$. Alors $W^K(\bQ_p\otimes K)$ s'identifie avec $K^{\times,w}_p$ et $W^K(\bA_\bQ)$ s'identifie hom\'eomorphiquement (pour la topologie produit) avec le produit restreint $\prod_{p\leq \infty}K^{\times,w}_p$ par rapport aux sous-groupes $\cO^{\times,w}_{K,p}$. Notons encore $W^K(\bA_\bQ)$ ce produit restreint. 

Soit donc $(a_p)\in W^K(\bA_\bQ)$ et soit $\cU$ un voisinage de $(a_p)$ dans $W^K(\bA_\bQ)$. Pour tout $p$, notons $q_p$ la valeur commune des $\N_{K/K_0}\alpha_v$, o\`u $a_p=(\alpha_v)_{v\in V_p}$. L'application $\N_{K/K_0}:(\bQ_p\otimes K)^\times \to (\bQ_p\otimes K_0)^\times$ \'etant une application ouverte qui envoie $(\bZ_p\otimes \cO_K)^\times$ sur $(\bZ_p\otimes \cO_{K_0})^\times$ lorsque $p$ est non ramifi\'e, $\N_{K/K_0}(\cU)$  est un voisinage de $(q_p)$. Il est bien connu que $\bG_{m,\bQ}$ satisfait \`a l'approximation faible: choisissons donc $q\in \bQ^\times\cap \N_{K/K_0}(\cU)$. Quitte \`a modifier les coefficients $a_p$ tout en restant dans $\cU$, on peut supposer que $q_p=q$ pour tout $p$. Alors $q$ est une $K/K_0$-norme locale partout: d'apr\`es un th\'eor\`eme de Hasse, il existe donc $\alpha_0\in K^\times$ tel que $\N_{K/K_0}\alpha_0=q$. On  interpr\`ete les solutions $\alpha$ de $N_{K/K_0}\alpha=q$ comme les points d'une conique \`a coefficients dans $K_0$. Pour tout $p$, soit $m_p$ la pente de la droite passant par $a_p$ et par $\alpha_0$. En utilisant l'approximation faible pour $\bG_{a,K_0}$, on trouve $m\in K$ tel que la droite de pente $m$ passant par $\alpha_0$ intersecte la conique en un point $\alpha$ appartenant \`a $K^\times\cap \cU$, d'o\`u le r\'esultat.
\end{proof}   

\section{Le cas o\`u $g=2$}{\label{sec:gendeux}}

Dans ce n$^o$, nous montrons le prolongement m\'eromorphe des fonctions $Z(K,\chi,s)$ et $\tilde{Z}(K,\chi,s)$ \`a tout le plan complexe lorsque $g=2$. Il est facile de voir que, quel que soit le nombre premier $p$, le facteurs eul\'eriens $Z_p(K,\chi,s)$ et $\tilde{Z}_p(K,\chi,s)$ sont des fonctions rationelles de $p^s$. Comme $Z_p(K,\chi,s)=\tilde{Z}_p(K,\chi,s)$ pour tout nombre premier non ramifi\'e dans $K$, il suffit de traiter le cas de $Z(K,\chi,s)$.

Si $f(s)=\prod_{p}f_p(s)$ est une fonction somme d'une s\'erie de Dirichlet qui s'\'ecrit comme produit de facteurs eul\'eriens $f_p$, on note $f^{\nr}(s)$ le produit des $f_p(s)$ lorsque $p$ parcourt l'ensemble des nombres premiers non ramifi\'es dans $K$.

Si $M$ est une extension quadratique de $\bQ$, on note $\ve_{M}$ le caract\`ere non-trivial du groupe de Galois de $M/\bQ$ et soit $L(\ve_M,s)$ la s\'erie de Dirichlet associ\'ee. Rappelons que
\begin{equation*}
L(\ve_M,s)=\prod_{p \text{ non ramifi\'e}}\left(1-\ve_M(p)p^{-s}\right)^{-1}
\end{equation*}
o\`u le produit parcourt l'ensemble des nombres premiers $p$ non-ramifi\'es dans $M$ et $\ve_M(p)=1$ ou $-1$ selon que l'id\'eal $p\cO_M$ est d\'ecompos\'e ou reste premier dans $M$.

\begin{prop}\label{prop:g2prolong}
Soit $K$ un corps CM quartique et soit $\chi:\cC^w_K\to \bC^\times$ un caract\`ere. 

\case{i} On suppose $K$ galoisien sur $\bQ$ avec groupe de Galois $\Gamma$ cyclique d'ordre $4$. Alors $\Gamma$ op\`ere transitivement sur $\cC\cM_K$. Soit $\Phi\in \cC\cM_K$. Alors 
\begin{equation*}
Z^{\nr}(K,\chi,s)=L^{\nr}(\chi_{\hat{\Phi}},s)L^{\nr}(\ve_{K_0},2s)^{-1},
\end{equation*}
o\`u $K_0$ est le sous-corps quadratique de $K$.

\case{ii} On suppose $K$ galoisien sur $\bQ$ avec groupe de Galois $\Gamma$ un groupe de Klein. Alors l'action de $\Gamma$ sur $\cC\cM_K$ poss\`ede deux orbites. Soient $\Phi_1$, $\Phi_2\in \cC\cM_K$ des repr\'esentants de ces deux orbites. Alors 
\begin{equation*}
Z^{\nr}(K,\chi,s)=L^{\nr}(\chi_{\hat{\Phi}_1},s)L^{\nr}(\chi_{\hat{\Phi}_2},s)\zeta^{\nr}(2s)^{-1}.
\end{equation*}

\case{iii} On suppose $K$ non-galoisien, le groupe $\Gamma$ d'une cl\^oture galoisienne $L$ \'etant un groupe di\'edral d'ordre $8$. Alors $\Gamma$ op\`ere transitivement sur $\cC\cM_K$. Soit $\Phi\in \cC\cM_K$ et soit $\hat{K}_0$ le sous-corps quadratique du corps reflex $\hat{K}$. Alors 
\begin{equation*}
Z^{\nr}(K,\chi,s)=L^{\nr}(\chi_{\hat{\Phi}},s)L^{\nr}(\ve_{\hat{K}_0},2s)^{-1}.
\end{equation*}
\end{prop}

Il est clair que la proposition implique le prolongement m\'eromorphe sur $\bC$ des fonctions $Z(K,\chi,s)$, car les fonctions $L$ apparaissant dans les membres droites des formules (compl\'et\'ees par leurs facteurs aux premiers ramifi\'es dans $K$) poss\`edent un tel prolongement. Par ailleurs, il est bien connu que tout corps CM quartique se trouve dans l'un des cas \case{i}, \case{ii} et \case{iii}.

\begin{proof} La description de l'action de $\Gamma$ sur $\cC\cM_K$ est facile et laiss\'ee au lecteur (voir la remarque~\ref{rem:rhoKcalcul}). Le reste de la d\'emonstration consiste en la comparaison des facteurs eul\'eriens associ\'ees aux nombres premiers $p$ non-ramifi\'ees dans $K$. Cette comparaison se fait cas par cas selon la nature de la d\'ecomposition de $p$ dans $K$. 

Prenons le cas \case{i}, le nombre premier $p$ \'etant suppos\'e totalement d\'ecompos\'e dans $K$. Alors $p$ est d\'ecompos\'e dans $K_0$ et $L_p(\ve_{K_0},s)=(1-p^{-s})^{-1}$. Par cons\'equent,
\begin{align*}
L_p(\ve_{K_0},2s)Z_p(K, \chi, s)&=(1-p^{-2s})^{-1}\sum_{k\geq 0}\sum_{\ga\in \cI_{p^k}}\chi(\ga) p^{-ks}\\
&=\sum_{k\geq 0}\sum_{\ell=0}^{\lfloor{k/2}\rfloor}\sum_{\ga\in \cI_{p^{k-2\ell}}}\chi(p^{2\ell}\ga)p^{-ks},
\end{align*}
car $\chi(p\cO_K)=1$. 

D'autre part, \`a conjugaison complexe pr\`es, il existe un g\'en\'erateur $\sigma$ de $\Gamma$ tel que $\Phi(\ga)=\ga\sigma^{-1}(\ga)$ quel que soit l'id\'eal $\ga$ de $K$. Alors $\hat{\Phi}(\ga)=\ga\sigma(\ga)$ et $\chi_{\hat{\Phi}}(\ga)=\chi(\ga\sigma(\ga))$. Soit $\gp$ l'un des id\'eaux premiers de $K$ divisant $p$. On en tire que
\begin{equation*}
L_p(\chi_{\hat{\Phi}},s)=\prod_{i=0}^3(1-\chi_{\hat{\Phi}}(\gp^{\sigma^i})p^{-s})^{-1}=\prod_{i=0}^3(1-\chi(\gp^{\sigma^i}\gp^{\sigma^{i+1}})p^{-s})^{-1}. 
\end{equation*} 
Si $k\in \bN$ on note $E_k$ l'ensemble des id\'eaux de $K$ de la forme $p^{2\ell}\ga$ avec $\ga\in \cI_{p^{k-2\ell}}$ et $0\leq \ell\leq \lfloor{k/2}\rfloor$.  Un calcul court montre que tout id\'eal appartenant \`a $E_k$ s'\'ecrit de fa\c{c}on unique comme produit de $k$ id\'eaux de la forme $\gp^{\sigma^i}\gp^{\sigma^{i+1}}$ avec $i\in \{0,1,2,3\}$, et que r\'eciproquement tout produit de cette forme appartient \`a $E_k$.  En d\'eveloppant $L_p(\chi_{\hat{\Phi}},s)$ en puissances de $p^{-s}$, on voit que le coefficient de $p^{-ks}$ co\"{\i}ncident avec celui du d\'eveloppement de $L_p(\ve_{K_0},2s)Z_p(K,\chi,s)$. D'o\`u l'\'egalit\'e $L_p(\ve_{K_0},2s)Z_p(K,\chi,s)=L_p(\chi_{\hat{\Phi}},s)$. 

Dans le cas \case{ii}, on peut choisir $\Phi_1(\ga)=\N_{K/K_1}(\ga)$ et $\Phi_2(\ga)=\N_{K/K_2}(\ga)$, o\`u $K_1$, $K_2$ d\'esignent les deux sous-corps quadratiques imaginaires de $K$.  Alors $\hat{\Phi}_i(\gA)=\gA_i\cO_K$ ($i=1$, $2$), quel que soit l'id\'eal $\gA$ de $K_i$. Dans le cas \case{iii}, le corps $\hat{K}$ est l'un des deux sous-corps CM quartiques de $L$ qui ne sont pas conjugu\'es \`a $K$, et la factorisation de $\gp$ dans $K$ d\'etermine celle dans $L$ et donc celle dans $\hat{K}$. En plus, $L$ contient un unique sous-corps quadratique $M$ qui n'est contenu dans aucune sous-corps CM. Si $\sigma$ est le g\'en\'erateur de $\Gamma^K$ et $\tau$ un g\'en\'erateur de $\Gamma^M$, alors $\sigma$ et $\tau$ engendrent $\Gamma$ et on peut supposer que $\Phi(\ga)=\gA$, o\`u $\gA\cO_L=\ga(\sigma\tau)(\ga)\cO_L$. Alors $\hat{K}$ est caract\'eris\'e par $\Gamma^{\hat{K}}=\{\id, \sigma\tau\}$ et $\hat{\Phi}(\gA)=\ga'$, o\`u $\ga'\cO_L=\gA\sigma(\gA)\cO_L$. 

Ces indications devraient permettre au lecteur de compl\'eter les d\'etails de la d\'emonstration.\end{proof}

\end{document}